\documentclass[a4paper]{article}

\usepackage{amsmath}
\usepackage{amssymb}
\usepackage{amsthm}
\usepackage{dsfont}
\usepackage{enumerate}
\usepackage{pifont}
\usepackage{theoremref}
\usepackage{tikz}
\usetikzlibrary{trees}
\usetikzlibrary{decorations.markings}

\newcommand{\mc}[1]{{\mathcal{#1}}}
\newcommand{\mf}[1]{{\mathfrak{#1}}}
\newcommand{\bb}[1]{{\mathbb{#1}}}

\DeclareMathOperator{\RE}{Re}
\DeclareMathOperator{\IM}{Im}
\renewcommand{\Re}{\RE}
\renewcommand{\Im}{\IM}

\DeclareSymbolFont{largesymbolsA}{U}{txexa}{m}{n}
\DeclareMathSymbol{\varprod}{\mathop}{largesymbolsA}{16}

\DeclareMathOperator{\dom}{dom}
\DeclareMathOperator{\mul}{mul}
\DeclareMathOperator{\ran}{ran}

\newlength{\maxlabwidth}

\numberwithin{equation}{section}
\swapnumbers
\theoremstyle{plain}
	\newtheorem{lemma}{Lemma}[section]
	\newtheorem{proposition}[lemma]{Proposition}
	\newtheorem{theorem}[lemma]{Theorem}
	\newtheorem{corollary}[lemma]{Corollary}
	\newtheorem{ntheoreM}[lemma]{}
\theoremstyle{definition}
	\newtheorem{definitioN}[lemma]{Definition}
\theoremstyle{remark}
	\newtheorem{remarK}[lemma]{Remark}
	\newtheorem{examplE}[lemma]{Example}
	\newtheorem{nremarK}[lemma]{}
\newcommand{\thlab}[1]{\thlabel{#1}\label{#1.}}
\renewcommand{\qedsymbol}{\raisebox{-2pt}{\large\ding{113}}}
\newcommand{\defendsymbol}{$\lozenge$}
\newcommand{\qedsymbolsave}{\qedsymbol}

\newenvironment{definition}{\begin{definitioN}}{
	\renewcommand{\qedsymbolsave}{\qedsymbol}\renewcommand{\qedsymbol}{\defendsymbol}
	\popQED{\qed}\renewcommand{\qedsymbol}{\qedsymbolsave}\end{definitioN}}
\newenvironment{remark}{\begin{remarK}}{
	\renewcommand{\qedsymbolsave}{\qedsymbol}\renewcommand{\qedsymbol}{\defendsymbol}
	\popQED{\qed}\renewcommand{\qedsymbol}{\qedsymbolsave}\end{remarK}}
\newenvironment{example}{\begin{examplE}}{
	\renewcommand{\qedsymbolsave}{\qedsymbol}\renewcommand{\qedsymbol}{\defendsymbol}
	\popQED{\qed}\renewcommand{\qedsymbol}{\qedsymbolsave}\end{examplE}}

\newcommand{\bibi}[5]{\bibitem[#5]{#1} \textsc{#2}:\ \textit{#3,}\ {#4.} }

\begin{document}
\begin{flushleft}
	{\Large\bf Definitizability of normal operators on Krein spaces and their functional calculus}
	\\[5mm]
	\textsc{Michael Kaltenb\"ack\footnote{This work was supported by a joint project of the Austrian Science Fund 
		  (FWF, I1536--N25) and the Russian Foundation for Basic Research (RFBR, 13-01-91002-ANF).}}
	\\[6mm]
	{\small
	\textbf{Abstract: We discuss a new concept of definitizability of a normal 
		operator on Krein spaces. For this new concept we develop
		a functional calculus $\phi \mapsto \phi(N)$ which is the 
		proper analogue of $\phi \mapsto \int \phi \, dE$ in the Hilbert space situation.}
	}
\end{flushleft}

\begin{flushleft}
\end{flushleft}
\begin{flushleft}
\end{flushleft}

\section{Introduction}

A bounded linear operator $N$ on a Krein space $(\mc K,[.,.])$ is normal, if
$N$ commutes with its Krein space adjoint $N^+$. If we write $N$ as $A+iB$
with the selfadjoint real part $A:=\Re N:=\frac{N+N^+}{2}$ and the selfadjoint
imaginary part $B:=\Im N:=\frac{N-N^+}{2i}$, then $N$ is normal if and only if $AB=BA$.
In \cite{Ka2015} we called a normal $N$ definitizable
whenever $A$ and $B$ were both definitizable in the classical sense,
i.e.\ there exist so-called definitizing polynomials $p(z),q(z)\in \bb R[z]\setminus\{0\}$ such that 
$[p(A)x,x] \geq 0$ and $[q(B)x,x] \geq 0$ for all $x\in \mc K$. 

For such definitizable operators in \cite{Ka2015} we could build a functional calculus
in analogy to the functional calculus $\phi\mapsto \int \phi \, dE$ 
mapping the $*$-algebra of bounded and measurable functions on $\sigma(N)$ 
to $B(\mc H)$ in the Hilbert space case. The functional calculus in \cite{Ka2015} can also be seen as
a generalization of Heinz Langers spectral theorem on definitizable selfadjoint operators
on Krein spaces; see \cite{langer1982}, \cite{KaPr2014}.
Unfortunately, there are unsatisfactory phenomenons with this concept of definitizability in \cite{Ka2015}.  
For example, it is not clear, whether for a bijective, normal definitizable $N$ also  
$N^{-1}$ definitizable.

In the present paper we choose a more general concept of definitizability. We shall say that a normal
$N$ on a Krein space $\mc K$ is definitizable if $[p(A,B)u,u] \geq 0$ for all $u \in \mc K$
for some, so-called definitizing, $p(x,y)\in\bb C[x,y]\setminus\{0\}$ with real coefficients.
Then we study the ideal $I$ generated by all definitizing polynomials with real coefficients in $\bb C[x,y]$, and 
assume that $I$ is large in the sense that it is zero-dimensional, i.e.\ $\dim \bb C[x,y]/I<\infty$.
By the way, if $N$ is definitizable in the sense of \cite{Ka2015}, then $I$ is always zero-dimensional.

Using results from algebraic geometry, under the assumption that $I$ is zero-dimensional, the 
variety $V(I)=\{a\in \bb C^2: f(a)=0 \ \text{ for all } \ f \in I\}$ is a finite set. We split
this subset of $\bb C^2$ up as
\[
    V(I) = (V(I)\cap \bb R^2) \dot \cup (V(I)\setminus \bb R^2) \,,
\]
and interpret $V_{\bb R}(I):=V(I)\cap \bb R^2$ in the following as a subset of $\bb C$ by
consider the first entry as the real and the second entry as the imaginary part.

By the ascending chain condition the ideal $I$ is generated by real definitizing polynomials
$p_1,\dots,p_m$. With the help of the positive semidefinite scalar products 
$[p_j(A,B).,.]$, $j=1,\dots,m$ and $\sum_{k=1}^m [p_k(A,B).,.]$ we construct Hilbert spaces 
$\mc H_j$, $j=1,\dots,m$ and $\mc H$ together with bounded and injective
$T_j: \mc H_j \to \mc K$ and $T: \mc H \to \mc K$. We consider
$\Theta_j: (T_jT_j^+)' \to (T_j^+T_j)'$ and $\Theta: (TT^+)' \to (T^+T)'$ by 
$\Theta_j(C):= (T_j\times T_j)^{-1}(C)$ and $\Theta(C):= (T\times T)^{-1}(C)$, as studied in \cite{KaPr2014}.
Here $(T_jT_j^+)', (TT^+)'\subseteq B(\mc K)$ and $(T_j^+T_j)'\subseteq B(\mc H_j), (T^+T)'\subseteq B(\mc H)$ denote 
the commutant of the respective operators.

The proper family $\mc F_N$ of functions suitable for the aimed functional 
calculus are functions defined on
\[
    \big(\sigma(\Theta(N)) \cup V_{\bb R}(I) \big) \dot\cup (V(I)\setminus \bb R^2) \,.
\]
Moreover, the functions $\phi\in \mc F_N$ assume values in $\bb C$ on 
$\sigma(\Theta(N)) \setminus V_{\bb R}(I)$ and values in a certain finite dimensional 
$*$-algebras $\mc A(z)$ at $z\in V_{\bb R}(I)$ and $\mc B((\xi,\eta))$ at $(\xi,\eta)\in V(I)\setminus \bb R^2$.
On $\sigma(\Theta(N)) \setminus V_{\bb R}(I)$ we assume $\phi$ to be bounded and measurable.
Finally, $\phi\in \mc F_N$ satisfies a growth regularity condition at all
$w$ points from $V_{\bb R}(I)$ which are not isolated in
$\sigma(\Theta(N)) \cup V_{\bb R}(I)$.
Vaguely speaking, this growth regularity condition means that around 
$w$ the function $\phi$ admits an approximation by a Taylor polynomial, which is determined by 
$\phi(w)\in \mc A(w)$.
Any polynomial $s(x,y)\in \bb C[x,y]$ can be seen as a function $s_N\in \mc F_N$ in a natural way.

For each $\phi\in\mc F_N$ we will see that there exists $p\in\bb C[x,y]$ and bounded, measurable
$f_1,\dots,f_m: \sigma(\Theta(N)) \cup V_{\bb R}(I) \to \bb C$ with 
$f_j(z)=0$ for $z\in V_{\bb R}(I)$ such that
\begin{equation}\label{decompact2}
  \phi(z) = p_N(z) + \sum_j f_j(z) \, (p_j)_N(z) 
\end{equation}
for all $z\in \sigma(\Theta(N)) \cup V_{\bb R}(I)$, and that
$\phi((\xi,\eta)) = p_N((\xi,\eta))$ for all $(\xi,\eta) \in V(I)\setminus \bb R^2$.
We then define 
\[
    \phi(N):=p(A,B) + \sum_{k=1}^m T_k \int_{\sigma(\Theta_k(N))} f_k \, dE \, T_k^+ \,, 
\]
and show that
this operator does not depend on the actual decomposition \eqref{decompact2} and 
that $\phi \mapsto \phi(N)$ is indeed a $*$-homomorphism satisfying $\phi(N) = s(A,B)$
for $\phi=s_N$.

\section{Multiple embeddings}

In the present section $(\mc K,[.,.])$ will be a Krein space and
$(\mc H,(.,.))$, $(\mc H_j,(.,.)), \ j=1,\dots,m$, will denote Hilbert spaces.
Moreover, let $T: \mc H \to \mc K$, $T_j: \mc H_j \to \mc K$ and
$R_j: \mc H_j \to \mc H$ bounded, linear and injective mappings such that $TR_j=T_j$.
By $T^+: \mc K \to \mc H$ and $T_j^+: \mc K \to \mc H_j$ we denote the respective Krein space adjoints.

If $D$ is an operator on a Krein space, then we shall denote by $D'$ the commutant
of $D$, i.e.\ the algebra of all operators commuting with $D$. For a selfadjoint $D$
this commutant is a $*$-algebra with respect to forming adjoint operators.

For $j=1,\dots,m$ we shall denote by $\Theta_j: (T_jT_j^+)' \ (\subseteq B(\mc K)) \to (T_j^+T_j)' \ (\subseteq B(\mc H_j))$, and
by $\Theta: (TT^+)' \ (\subseteq B(\mc K)) \to (T^+T)' \ (\subseteq B(\mc H))$ the
$*$-algebra homomorphisms mapping the identity operator to the identity operator 
as in \thref{thetadefeig} from \cite{KaPr2014} corresponding to the mappings $T_j$ and $T$:
\[
    \Theta_j(C_j) = (T_j\times T_j)^{-1}(C_j) = T_j^{-1}C_jT_j, \ \ C_j\in (T_jT_j^+)' \,,
\]    
\begin{equation}\label{thetaVdef}
    \Theta(C) = (T\times T)^{-1}(C) = T^{-1}CT, \ \ C\in (TT^+)' \,.
\end{equation}
We can apply \thref{thetadefeig} in \cite{KaPr2014} also to the bounded linear, injective  
$R_j: \mc H_j \to \mc H$, and denote the corresponding $*$-algebra homomorphisms
by $\Gamma_j : (R_jR_j^*)' \ (\subseteq B(\mc H)) \to (R_j^*R_j)' \ (\subseteq B(\mc H_j))$:
\[
    \Gamma_j(D) = (R_j\times R_j)^{-1}(D) = R_j^{-1}DR_j, \ \ D \in (R_jR_j^*)' \,.
\]
For the following note that due to $(\ran T^+)^{[\bot]} = \ker T = \{0\}$ the range of $T^+$ is dense in $\mc H$.

\begin{lemma}\thlab{comreg}
    For $j=1,\dots,m$ we have 
    $\Theta((T_jT_j^+)'\cap (TT^+)') \subseteq (R_jR_j^*)' \cap (T^+T)'$,
    where in fact
\begin{equation}\label{zuef}
  \Theta(C) R_j R_j^* = R_j \Theta_j(C) R_j^*=R_j R_j^* \Theta(C), \ \ C\in (T_jT_j^+)'\cap (TT^+)'  \,.
\end{equation}
    Moreover,
\begin{equation}\label{zuefvor}
  \Theta_j(C) = \Gamma_j \circ \Theta(C), \ \ C\in (T_jT_j^+)'\cap (TT^+)'  \,.
\end{equation}
\end{lemma}
\begin{proof}
    According to \thref{thetadefeig} in \cite{KaPr2014} we have 
    $\Theta_j(C) T_j^+ = T_j^+ C$ and $T^+ C = \Theta(C) T^+$ for $C\in (T_jT_j^+)'\cap (TT^+)'$.
    Therefore,
\begin{align*}
    T (\, R_j \Theta_j(C) R_j^* \, ) T^+ & =  T_j \Theta_j(C) T_j^+ = T_j T_j^+ C \\ & =
    T R_j R_j^* T^+ C = T ( \, R_j R_j^* \Theta(C) \, ) T^+ \,. 
\end{align*}
$\ker T=\{0\}$ and the density of $\ran T^+$ yield $R_j \Theta_j(C) R_j^*=R_j R_j^* \Theta(C)$.
Applying this equation to $C^+$ and taking adjoints yields 
$R_j \Theta_j(C) R_j^*=\Theta(C) R_j R_j^*$. In particular, $\Theta(C) \in (R_jR_j^*)'$.
Therefore, we can apply $\Gamma_j$ to $\Theta(C)$ and get
\[
  \Gamma_j \circ \Theta(C) = R_j^{-1} T^{-1}C T R_j = T_j^{-1} C T_j = \Theta_j(C) \,.
\]
\end{proof}

For the following assertion note that by \eqref{zuefvor} and by the fact that $\Gamma_j$
is a $*$-algebra homomorphism mapping the identity operator to the identity operator,
for $j=1,\dots,m$ we have 
\begin{equation}\label{4ucd75}
  \sigma(\Theta(C)) \subseteq \sigma(\Theta_j(C)) \ \ \text{ for all } \ \ C\in (T_jT_j^+)'\cap (TT^+)' \,.
\end{equation}

\begin{corollary}\thlab{normtransf}
     For a $j\in\{1,\dots,m\}$ let $N\in B(\mc K)$ be normal such that $N\in(T_jT_j^+)'\cap (TT^+)'$.
     Then $\Theta(N)$ is a normal operator in the Hilbert space $\mc H$, and 
     $\Theta_j(N)$  is a normal operator in the Hilbert space $\mc H_j$.
     Denoting by $E$ ($E_j$) the spectral
     measure of $\Theta(N)$ ($\Theta_j(N)$), we have $E(\Delta) \in (R_jR_j^*)' \cap (T^+T)'$ and
     \[
	\Gamma_j(E(\Delta)) = E_j(\Delta) \,,
     \]
     for all Borel subsets $\Delta$ of $\bb C$, where $E_j(\Delta) \in (R_j^*R_j)' \cap (T_j^+T_j)'$. 
     
     Moreover, $\int h \, dE \in (R_jR_j^*)' \cap (T^+T)'$ 
     and 
     \[
	\Gamma_j\left(\int h \, dE \right) = \int h \, dE_j 
     \]
     for any bounded and measurable $h: \sigma(\Theta(N)) \to \bb C$, where $\int h \, dE_j 
     \in (R_j^*R_j)' \cap (T_j^+T_j)'$.
\end{corollary}
\begin{proof}
    The normality of $\Theta(N)$ and $\Theta_j(N)$ is clear, since $\Theta$ and $\Theta_j$ are
    $*$-homomorphisms. From \thref{comreg} we know that $\Theta(N) \in (R_jR_j^*)' \cap (T^+T)'$.
    According to the well known properties of $\Theta(N)$'s spectral measure we obtain
    $E(\Delta) \in (R_jR_j^*)' \cap (T^+T)'$ and, in turn, 
    $\int h \, dE \in (R_jR_j^*)' \cap (T^+T)'$. In particular, $\Gamma_j$ can be applied 
    to $E(\Delta)$ and $\int h \, dE$.
    Similarly, $\Theta_j(N)\in (T_j^+T_j)'$ implies $E_j(\Delta), \int h \, dE_j \in (T_j^+T_j)'$ 
    for a bounded and measurable $h$.
    
    Recall from \thref{thetadefeig} in \cite{KaPr2014} that $\Gamma_j(D) R_j^*x = R_j^* D$
    for $D \in (R_jR_j^*)'$. Hence, for $x\in \mc H$ and $y\in \mc H_j$ we have
    \[
	(\Gamma_j(E(\Delta)) R_j^*x, y) = (R_j^* E(\Delta) x,y) = (E(\Delta) x, R_j y)
    \]
    and, in turn,
    \begin{multline*}
	\int_{\bb C} s(z,\bar z) \, d(\Gamma_j(E) R_j^*x, y) = 
	\int_{\bb C} s(z,\bar z) \, d(E x, R_j y) = (s(\Theta(N),\Theta(N)^*)x, R_j y) \\
	= (R_j^* s(\Theta(N),\Theta(N)^*)x, y) = (\Gamma_j\big(s(\Theta(N),\Theta(N)^*)\big) R_j^* x,y) 
    \end{multline*}
    for any $s(z,w) \in \bb C[z,w]$.
    By \eqref{zuefvor} and the fact, that $\Gamma_j$ is a $*$-homomorphism, 
    we have $\Gamma_j(s(\Theta(N),\Theta(N)^*)) = s(\Theta_j(N),\Theta_j(N)^*)$.
    Consequently,
    \[
	\int_{\bb C} s(z,\bar z) \, d(\Gamma_j(E) R_j^*x, y) = \int_{\bb C} s(z,\bar z) \, d(E_j R_j^*x, y) \,.
    \]
    Since $E(\bb C\setminus K)=0$ and $E_j(\bb C\setminus K)=0$ for a certain compact $K\subseteq \bb C$ and since 
    the set of all $s(z,\bar z), \ s\in \bb C[z,w]$, is densely contained in $C(K)$, we obtain from the uniqueness assertion 
    in the Riesz Representation Theorem 
    \[
	(\Gamma_j(E(\Delta)) R_j^*x, y) = (E_j(\Delta) R_j^*x, y) \ \ \text{ for all } \ \ x\in \mc H, \, y\in \mc H_j \,,
    \]
    for all Borel subsets $\Delta$ of $\bb C$.
    Due to the density of $\ran R_j^*$ in $\mc H_j$ we even have 
    $(\Gamma_j(E(\Delta)) z, y) = (E_j(\Delta) z, y)$ for all $y,z \in \mc H_j$, and in turn
    $\Gamma_j(E(\Delta))=E_j(\Delta)$. Since $\Gamma_j$ maps into $(R_j^*R_j)'$, we have
    $E_j(\Delta)\in (R_j^*R_j)'$. This yields  $\int h \, dE_j\in (R_j^*R_j)'$ for any bounded and measurable $h$.
    
    If $h: \sigma(\Theta(N)) \to \bb C$ is bounded and measurable, then by \eqref{4ucd75} also
    its restriction to $\sigma(\Theta_j(N)) = \sigma(\Gamma_j\circ\Theta(N))$ is bounded and measurable. 
    Due to $E_j(\Delta) R_j^*= \Gamma_j(E(\Delta)) R_j^* = R_j^*E(\Delta)$, for 
    $x\in\mc H$ and $y\in\mc H_j$ we have
    \begin{multline*}
	(\Gamma_j\left(\int h \, dE \right) R_j^*x,y) = 
	(R_j^* \left(\int h \, dE \right) x,y) = 
	(\left(\int h \, dE \right) x, R_j y) \\
	= \int h \, d(E x,R_j y) = \int h \, d(E_j R_j^* x, y) =
	(\left(\int h \, dE_j\right) R_j^*x,y) \,. 
    \end{multline*}
    The density of $\ran R_j^*$ yields $\Gamma_j\left(\int h \, dE \right) = \int h \, dE_j$.
\end{proof}

Recall from \thref{Xidefeig} in \cite{KaPr2014} the mappings ($j=1,\dots,m$)
\[
    \Xi_j : B(\mc H_j) \to B(\mc K), \ \
		\Xi_j(D_j) = T_j D_j T_j^{+} \,,
\]
and $\Xi : B(\mc H) \to B(\mc K)$ with $\Xi(D) = T D T^{+}$.
By ($j=1,\dots,m$)
\[
    \Lambda_j: B(\mc H_j) \to B(\mc H), 
	\ \ \Lambda_j(D_j) = R_j D_j R_j^{*} \,,
\]
we shall denote the corresponding mappings outgoing from the mappings $R_j: \mc H_j \to \mc H$. 
Due to $T_j=T R_j$ we have $\Xi_j = \Xi \circ \Lambda_j$.

According to \thref{Xidefeig} in \cite{KaPr2014}, $\Lambda_j\circ \Gamma_j(D) = D R_jR_j^*$ for
$D\in (R_jR_j^*)'$.
Hence, using the notation from \thref{normtransf}
\begin{equation}\label{zuef2}
    \Xi_j(\int h \, dE_j) = \Xi\Big(\Lambda_j\circ \Gamma_j\left(\int h \, dE \right)\Big) = 
    \Xi(R_jR_j^{*} \int h \, dE) \,.
\end{equation}

\begin{lemma}\thlab{existtreanspost}
    Assume that for $j\in\{1,\dots,m\}$ the operator $T_jT_j^+$ commutes with $TT^+$ on $\mc K$.
    Then the operators $R_j R_j^*$, $T^+ T$ commute on $\mc H$ and $R_j^*R_j$, $T_j^+ T_j$ commute on $\mc H_j$. Moreover, 
    \begin{equation}\label{zuef3}
	\Theta(T_jT_j^+) = R_j R_j^* T^+ T = T^+ T R_j R_j^* \,.
    \end{equation}
\end{lemma}
\begin{proof}
If $T_jT_j^+$ and $TT^+$ commute on $\mc K$, then
\[
    T ( \, T^+ T R_j R_j^* \, )T^+ = TT^+ T_jT_j^+ = T_jT_j^+ TT^+ = T ( \, R_j R_j^* T^+ T \, )T^+ \,.
\]
Employing $T$'s injectivity and the density of $\ran T^+$, we see that $R_j R_j^*$ and $T^+ T$ commute. 
From this we derive
\[
    T_j^+ T_j R_j^*R_j = R_j^* (T^+ T R_j R_j^*) R_j = R_j^* (R_j R_j^* T^+ T) R_j = R_j^*R_j T_j^+ T_j \,.
\]
\eqref{zuef3} follows from
\[
    T^{-1} T_jT_j^+ T = T^{-1} T R_j R_j^* T^+ T = R_j R_j^* T^+ T \,.
\]
\end{proof}

\section{Definitizability} \label{definitiza}

In \cite{Ka2015} we said that a normal $N\in B(\mc K)$ is definitizable, if its real part 
$A:=\frac{N+N^+}{2}$ and its imaginary part $B:=\frac{N-N^*}{2i}$ are definitizable 
in the sense that there exist real polynomials $p,q\in\bb R[z]\setminus \{0\}$ such that 
$[p(A)v,v] \geq 0$ and $[q(B)v,v] \geq 0$ for all $v\in \mc K$.
In the present note we will relax this condition. 

\begin{definition}\thlab{definidef}
For a normal $N\in B(\mc K)$ we call $p(x,y) \in \bb C[x,y]\setminus \{0\}$ a definitizing
polynomial for $N$, if 
\begin{equation}\label{defalgl}
    [p(A,B) v,v] \geq 0 \ \ \text{ for all } \ \ v\in \mc K \,.
\end{equation}
where $A=\frac{N+N^+}{2}$ and $B=\frac{N-N^+}{2i}$. If such a definitizing
$p \in \bb C[x,y]\setminus \{0\}$ exists, then we call $N$ definitizable normal.
\end{definition}

Clearly, we could also write $p$ as a polynomial of the variables $N$ and $N^+$.
But because of $A=A^+$ and $B=B^+$, writing $p$ as a polynomial of the 
variables $A$ and $B$ has some notational advantages.

\begin{remark}\thlab{definidefrem}
    According to \eqref{defalgl} the operator $p(A,B)\in B(\mc K)$ must be selfadjoint;
    i.e.\ $p(A,B)^+ = p^\#(A,B)$, where $p^\#(x,y) = \overline{p(\overline{x},\overline{y})}$.
    Hence, $q:=\frac{p_j+ p_j^\#}{2}$ is real, i.e.\ $q(x,y) \in \bb R[x,y]\setminus \{0\}$, and     
    satisfies $q(A,B) = p(A,B)$.
    Thus, we can assume that a definitizing polynomial is real.
\end{remark}

In the present section we assume that $p_j(x,y) \in \bb R[x,y]\setminus \{0\}$, $j=1,\dots,m$, are real, 
definitizing polynomial for $N$.

\begin{proposition}\thlab{defNspaces}
  With the above assumptions and notation
  there exist Hilbert spaces $(\mc H,(.,.))$, $(\mc H_j,(.,.))$, $j=1,\dots,m$ and 
  bounded linear and injective operators $T: \mc H \to \mc K$, $T_j: \mc H_j \to \mc K$,
  such that 
  \[
      T_j T_j^+ = p_j(A,B), \ \ \text{ and } \ \ 
      T T^+ = \sum_{k=1}^m T_k T_k^+ = \sum_{k=1}^m p_k(A,B) \,.  
  \]
\end{proposition}
\begin{proof}
Let $(\mc H_j,(.,.))$ be the Hilbert space completion of $\mc K/\ker p_j(A,B)$ with respect to 
$[p_j(A,B).,.]$ and let $T_j: \mc H_j \to \mc K$
be the adjoint of the factor mapping $x\mapsto x+ \ker p_j(A,B)$ of $\mc K$ into $\mc H_j$. 
Since $T_j^+$ has dense range, $T_j$ must be injective.
Similarly, let 
$(\mc H,(.,.))$ be the Hilbert space completion of $\mc K/(\ker \sum_{k=1}^m p_k(A,B))$ with respect 
to $[\big(\sum_{k=1}^m p_k(A,B)\big).,.]$ and let
$T: \mc H \to \mc K$ be the injective adjoint of the factor mapping of $\mc K$ into $\mc H$. 

From $[T T^+ x,y] = (T^+ x,T^+ y) = (x,y) = [\big(\sum_{k=1}^m p_k(A,B)\big) x,y]$ and
$[T_j T_j^+ x,y] = (T_j^+ x,T_j^+ y) = (x,y) = [p_j(A,B)x,y]$ 
for all $x,y\in\mc K$ we conclude
  \[
      T_j T_j^+ = p_j(A,B) \ \ \text{ and } \ \ T T^+ = \sum_{k=1}^m p_k(A,B) \,,
  \]
where the operators $T_j T_j^+ = p_j(A,B)$, $j=1,\dots,m$, pairwise commute, because $A$ and $B$ do.
\end{proof}

\thref{defNspaces} in particular yields
\begin{equation}\label{ghqw73si}
    TT^+ = \sum_{k=1}^m T_kT_k^+ 
\end{equation}
Since for $x\in \mc K$ and $j\in \{1,\dots,m\}$ we have
\[ 
	(T^+x,T^+x) = [TT^+x,x] = \sum_{k=1}^m [T_kT_k^+x,x] = 
		  \sum_{k=1}^m (T_k^+x,T_k^+x) \geq (T_j^+x,T_j^+x)  \,,
\] 
one easily concludes that $T^+ x \mapsto T_j^+x$ constitutes a well-defined, contractive linear mapping from 
$\ran T^+$ onto $\ran T_j^+$. By $(\ran T^+)^\bot = \ker T=\{0\}$ and $(\ran T_j^+)^\bot = \ker T_j=\{0\}$ 
these ranges are dense in the Hilbert spaces $\mc H$ and $\mc H_j$. Hence, 
there is a unique bounded linear continuation of $T^+ x \mapsto T_j^+x$ to $\mc H$, 
which has dense range in $\mc H_j$.

Denoting by $R_j$ the adjoint mapping of this continuation we clearly have $T_j = T R_j$
and $\ker R_j \subseteq \ker T_j = \{0\}$. From \eqref{ghqw73si} we conclude
\[
      T( \, I_{\mc H} \, )T^+ = TT^+ = \sum_{k=1}^m T R_k R_k^+T^+ = T( \, \sum_{k=1}^m R_k R_k^+ \, )T^+ \,.
\]
$\ker T=\{0\}$ and the density of $\ran T^+$ yield $\sum_{k=1}^m R_k R_k^* = I_{\mc H}$.

\begin{lemma}\thlab{existtreans}
    With the above notations and assumptions for $j=1,\dots,m$ there exist injective contractions $R_j: \mc H_j \to \mc H$
    such that $T_j = T R_j$ and $\sum_{k=1}^m R_k R_k^* = I_{\mc H}$. Moreover, we have
    \begin{equation}\label{zfv6lkc}
	  \{N,N^+\}'=\{A,B\}' \subseteq
          \bigcap_{k=1,\dots,m} (T_kT_k^+)' \subseteq (TT^+)'
    \end{equation}
    for all $j\in\{1,\dots,m\}$. Finally,
\begin{equation}\label{heaybab}
\begin{aligned}
    p_j(\Theta(A),\Theta(B)) & = R_j R_j^* \big( \sum_{k=1}^m p_k(\Theta(A),\Theta(B)) \big) \\ & =
    \big( \sum_{k=1}^m p_k(\Theta(A),\Theta(B)) \big) R_j R_j^* 
    \,,
\end{aligned}
\end{equation}
and for any $u \in \bb C[x,y]$
\begin{equation}\label{tt2479}
     p_j(A,B)u(A,B) = \Xi_j\big(u(\Theta_j(A),\Theta_j(B))\big) = \Xi\big(R_jR_j^{*} u(\Theta(A),\Theta(B))\big) \,, 
\end{equation}
    where $\Theta: (TT^+)' \ (\subseteq B(\mc K)) \to (T^+T)' \ (\subseteq B(\mc H))$ is as in
  \eqref{thetaVdef} and $\Xi : B(\mc H) \to B(\mc K)$ with $\Xi(D) = T D T^{+}$. 
\end{lemma}
\begin{proof}
    The first part was shown above, and \eqref{zfv6lkc} is clear from \thref{defNspaces}.  

      From \eqref{zuef3} and \thref{thetadefeig} in \cite{KaPr2014} we get
\begin{align*}
      p_j(\Theta(A),\Theta(B)) & = \Theta(p_j(A,B)) = \Theta(T_jT_j^+) = R_j R_j^* \, T^+ T = R_j R_j^* \, \Theta(TT^+) \\ & =
      R_j R_j^* \, \Theta( \sum_{k=1}^m p_k(A,B) ) =  R_j R_j^* \, \big( \sum_{k=1}^m p_k(\Theta(A),\Theta(B)) \big) \,, 
\end{align*}
      where $R_j R_j^*$ commutes with $T^+ T= \sum_{k=1}^m p_k(\Theta(A),\Theta(B))$ by \thref{existtreanspost}.
      Finally, \eqref{tt2479} follows from (see \thref{Xidefeig} in \cite{KaPr2014})
\begin{align*}
     p_j(A,B)u(A,B) & = \Xi_j\big(\Theta_j(u(A,B))\big) = \Xi\circ\Lambda_j\circ\Gamma_j\big(\Theta(u(A,B))\big) \\ 
     \nonumber & =
    \Xi\big(R_jR_j^{*} u(\Theta(A),\Theta(B))\big) \,. 
\end{align*}
\end{proof}

By \eqref{zfv6lkc} we can apply \thref{normtransf} in the present situation. In particular, $\Theta(N)$
is a normal operator on the Hilbert space $\mc H$.
Condition \eqref{defalgl} for $p=p_j, \ j=1,\dots,m$, implies certain spectral properties of $\Theta(N)$.

\begin{lemma}\thlab{speknorm}
  With the above assumptions and notation for $j\in\{1,\dots,m\}$ we have
\[
    \{z\in \bb C : |p_j(\Re z,\Im z)| > \| R_j R_j^* \| \cdot |\sum_{k=1}^m p_k(\Re z,\Im z)| \} \subseteq \rho(\Theta(N)) \,.
\]
In particular, the zeros of $\sum_{k=1}^m p_k(\Re z,\Im z)$ in $\bb C$ are contained in 
$\rho(\Theta(N)) \cup \{z\in \bb C: p_j(\Re z,\Im z)=0 \ \text{ for all } \ j=1,\dots,m\}$.
\end{lemma}
\begin{proof}
Let $n\in \bb N$ and set 
\[
    \Delta_n:= \{z\in \bb C : |p_j(\Re z,\Im z)|^2 > \frac{1}{n} + \| R_j R_j^* \|^2 \cdot |\sum_{k=1}^m p_k(\Re z,\Im z)|^2 \} \,.
\]
For $x\in E(\Delta_n)(\mc H)$, where $E$ denotes $\Theta(N)$'s special measure, we then have
\begin{multline*}
    \| p_j(\Theta(A),\Theta(B)) x \|^2 = \int_{\Delta_n} |p_j(\Re \zeta,\Im \zeta)|^2 \, d(E(\zeta)x,x) \geq \\ 
	\int_{\Delta_n} \frac{1}{n} \, d(E(\zeta)x,x) +
	\| R_j R_j^* \|^2 \int_{\Delta_n} |\sum_{k=1}^m p_k(\Re \zeta,\Im \zeta)|^2 \, d(E(\zeta)x,x) \\ \geq
	\frac{1}{n} \|x\|^2 + \| R_j R_j^* \big(\sum_{k=1}^m p_k(\Theta(A),\Theta(B))\big)x \|^2 \,.
\end{multline*}
By \eqref{heaybab} this inequality can only hold for $x=0$. Since $\Delta_n$ is open,
by the Spectral Theorem for normal operators on Hilbert spaces we have $\Delta_n \subseteq \rho(N)$.
The asserted inclusion now follows from 
\[
  \{z\in \bb C : |p_j(\Re z,\Im z)| > \| R_j R_j^* \| \cdot |\sum_{k=1}^m p_k(\Re z,\Im z)| \} = 
  \bigcup_{n\in\bb N} \Delta_n \,.
\]
\end{proof}

In the following let $I$ the ideal $\langle p_1,\dots,p_m\rangle$ generated by 
the real definitizing polynomials $p_1,\dots,p_m$ in the ring $\bb C[x,y]$.
The variety $V(I)$ is the set of all common zeros $a=(a_1,a_2) \in \bb C^2$ of all
$p\in I$. Clearly, $V(I)$ coincides with the set of all $a\in \bb C^2$ such that
$p_1(a_1,a_2) = \dots = p_m(a_1,a_2) = 0$. $V_{\bb R}(I)$ is the set of all $a\in \bb R^2$,
which belong to $V(I)$. It is convenient for our purposes, to 
consider $V_{\bb R}(I)$ as a subset of $\bb C$:
\begin{align}\label{nullstmereal}
    V_{\bb R}(I) :&=
      \{z\in \bb C: f(\Re z,\Im z) = 0 \ \text{ for all } \ f\in I\} 
      \\ & \nonumber = \{z\in \bb C: p_k(\Re z,\Im z) = 0 \ \text{ for all } \ k \in \{1,\dots,m\}\} \,.
\end{align}

\begin{corollary}\thlab{korvda}
    Let $E$ denote the special measure of $\Theta(N)$. Then we have
\[
	R_jR_j^* \, E(\bb C \setminus V_{\bb R}(I)) = E(\bb C \setminus V_{\bb R}(I)) \, R_jR_j^* =
	\int_{\bb C \setminus V_{\bb R}(I)} \frac{p_j(\Re z,\Im z)}{\sum_{k=1}^m p_k(\Re z,\Im z)} \, dE(z) \,. 
\]
\end{corollary}
\begin{proof}
    First note that the integral on the right hand side exists as a bounded operator, because 
    by \thref{speknorm} we have $|p_j(\Re z,\Im z)| \leq \| R_j R_j^* \| \cdot |\sum_{k=1}^m p_k(\Re z,\Im z)|$
    for $z\in\sigma(\Theta(N))$.
    The first equality is known from \thref{normtransf}.

    Concerning the second equality, note that both sides vanish on the range of 
    $E(V_{\bb R}(I))$. 
    Its orthogonal complement $\mc Q:= \ran E(\bb C \setminus V_{\bb R}(I))$ is invariant under
    \[
	\int \big(\sum_{k=1}^m p_k(\Re z,\Im z)\big) \, dE(z) = \sum_{k=1}^m p_k(\Theta(A),\Theta(B)) \,.
    \]
    By \thref{speknorm} the restriction of this operator to $\mc Q$
    is injective, and hence, has dense range in $\mc Q$.
    If $x$ belongs to this dense range, i.e.\ 
    $x=\big(\sum_{k=1}^m p_k(\Theta(A),\Theta(B))\big) y$ with $y\in \mc Q$, then
    \begin{align*}
	\int_{\bb C \setminus V_{\bb R}(I)} & \frac{p_j(\Re z,\Im z)}{\sum_{k=1}^m p_k(\Re z,\Im z)} \, dE(z) x = 
	\int_{\bb C \setminus V_{\bb R}(I)} p_j(\Re z,\Im z) \, dE(z) y \\ & =
	p_j(\Theta(A),\Theta(B))y = R_j R_j^* \big(\sum_{k=1}^m p_k(\Theta(A),\Theta(B))\big) y = R_j R_j^* x \,.
\end{align*}
    By a density argument the second asserted equality of the present corollary holds true on $\mc Q$
    and in turn on $\mc H$. 
\end{proof}

\begin{remark}\thlab{nullrem}
    In \thref{defNspaces} the case that $p_j(A,B) = 0$ for some $j$, or even for all $j$, is not excluded, and 
    yields $\mc H_j=\{0\}$, $T_j=0$ and $R_j=0$ (in \thref{existtreans}), or even $\mc H=\{0\}$ and $T=0$.
    Also the remaining results hold true, if we interpret $\rho(R)$ as $\bb C$ and $\sigma(R)$ as $\emptyset$ 
    for the only possible linear operator $R=(0\mapsto 0)$ on the vector space $\{0\}$.
\end{remark}

\section{An Abstract Functional Calculus}
\label{abstrfunccal}

In this section let $\mc K$ be again a Krein space, $N\in B(\mc K)$ be a definitizable normal operator.
Let $I$ be the ideal in $\bb C[x,y]$, which is generated by all real definitizing polynomials.
By the ascending chain condition for the ring $\bb C[x,y]$ (see for example \cite{CLO1}, 
Theorem 7, Chapter 2, \S5) $I$ is generated by finitely many
real definitizing polynomials $p_1,\dots,p_m$, i.e.\ $I = \langle p_1,\dots,p_m \rangle$.
We employ the same notion as in the previous sections for these polynomials
$p_1,\dots,p_m$. In particular, $E_j$ ($E$) denotes the spectral measure
of $\Theta_j(N)$ on $\mc H_j$ ($\Theta(N)$ on $\mc H$).

We also make the convention that for $p\in \bb C[x,y]$ and $z\in \bb C$ we write
$p(z)$ short for $p(\Re z,\Im z)$.

\begin{lemma}\thlab{lre0}
    For a bounded and measurable $f: \sigma(\Theta(N)) \to \bb C$ and 
    $j\in \{1,\dots,m\}$ we have
    \begin{multline*}
	\Xi_j\left( \int_{\sigma(\Theta_j(N))} f \, dE_j \right) = \\ 
	\Xi\left( \int_{\sigma(\Theta(N))\setminus V_{\bb R}(I)} f\frac{p_j}{\sum_{l=1}^m p_l} \, dE 
	    + R_jR_j^* \int_{\sigma(\Theta(N))\cap V_{\bb R}(I)} f \, dE \right) \,. 
    \end{multline*}
\end{lemma}
\begin{proof}
    By \eqref{zuef2} the left hand side coincides with
    \[
	\Xi\left( R_jR_j^* \int_{\sigma(\Theta(N)) \setminus V_{\bb R}(I)} f \, dE 
		  + R_jR_j^* \int_{\sigma(\Theta(N)) \cap V_{\bb R}(I)} f \, dE \right) 
	\,.
    \]
    $\int_{\sigma(\Theta(N)) \setminus V_{\bb R}(I)} f \, dE = 
      E(\bb C \setminus V_{\bb R}(I)) \int_{\sigma(\Theta(N)) \setminus V_{\bb R}(I)} f \, dE$
    together with \thref{korvda} prove the equality. 
\end{proof}

\begin{lemma}\thlab{lre1}
    Let $f,g: \sigma(\Theta(N)) \to \bb C$ be bounded and measurable,
    and let $r\in \bb C[x,y]$. For $j,k\in \{1,\dots,m\}$ 
    we then have
    \begin{align}\label{rthz23}
	r(A,B) \, \Xi_j( \int_{\sigma(\Theta_j(N))} f \, dE_j ) & = 
	\Xi_j( \int_{\sigma(\Theta_j(N))} f \, dE_j ) \, r(A,B) \\ \nonumber & =
	\Xi_j( \int_{\sigma(\Theta_j(N))} r f \, dE_j ) \,,
    \end{align}
    and
    \begin{equation}\label{rthz23b}
	\Xi_j\left(\int_{\sigma(\Theta_j(N))} f \, dE_j \right)
	      \Xi_k\left(\int_{\sigma(\Theta_k(N))} g \, dE_k\right)
    \end{equation}
    \[
	= \Xi\left(\int_{\sigma(\Theta(N))} fg \, \frac{p_jp_k}{\sum_{l=1}^m p_l} \, dE\right)
    \]
    \[
	= \Xi_j\left(\int_{\sigma(\Theta_j(N))} fg \, p_k \, dE_j \right) = 
	\Xi_k\left(\int_{\sigma(\Theta_k(N))} fg \, p_j \, dE_k \right) \,. 
    \]
\end{lemma}
\begin{proof}
    By \thref{Xidefeig} in \cite{KaPr2014} we have
    \[
    r(A,B)\,\Xi_j(D)=\Xi_j(\Theta(r(A,B)) D) = \Xi_j(r(\Theta_j(A),\Theta_j(B)) D)\,,
    \]
    \[
    \Xi_j(D)r(A,B)=\Xi_j( D \, \Theta_j(r(A,B))) = \Xi_j(D \, r(\Theta_j(A),\Theta_j(B)))
    \]
    for $D\in (T^+ T)'$. For $D=\int_{\sigma(\Theta_j(N))} f \, dE_j$ this 
    implies \eqref{rthz23}.
    
    According to \eqref{zuef2} the expression in \eqref{rthz23b} coincides with
    \[
	\Xi\left(R_jR_j^* \int_{\sigma(\Theta(N))} f \, dE \right)
	      \Xi\left(R_kR_k^* \int_{\sigma(\Theta(N))} g \, dE\right) \,.
    \]
    By \thref{Xidefeig} and \thref{thetadefeig} in \cite{KaPr2014}, we also know that 
    $\Xi(D_1)\Xi(D_2)=\Xi(T^+T D_1D_2)=\Xi(\Theta(TT^+) D_1D_2)$, where  
    (see \thref{defNspaces} and \eqref{nullstmereal})
    \[
	\Theta(TT^+) = \sum_{l=1}^m p_l(\Theta(A),\Theta(B)) = \int \sum_{l=1}^m p_l \, dE =
	(\int \sum_{l=1}^m p_l \, dE) \, E(\bb C \setminus V_{\bb R}(I)) \,.
    \]
    Therefore, by \thref{korvda} and the fact, that 
    $E(\bb C \setminus V_{\bb R}(I))$ commutes with $\int_{\sigma(\Theta(N))} f \, dE$, 
    \eqref{rthz23b} can be written as
    \[
	\Xi\left( (\int \sum_{l=1}^m p_l \, dE) 
	    (\int \frac{p_j}{\sum_{l=1}^m p_l} \, dE) (\int f \, dE)
	    (\int \frac{p_k}{\sum_{l=1}^m p_l} \, dE) (\int g \, dE)\right) =
    \]
    \[
	\Xi\left(\int_{\sigma(\Theta(N))} fg \, \frac{p_jp_k}{\sum_{l=1}^m p_l} \, dE\right) \,.
    \]
    The remaining equalities follow from \thref{lre0} since the respective integrands
    vanish on $V_{\bb R}(I)$.
\end{proof}

\begin{lemma}\thlab{lre2}
    For a bounded and measurable $f: \sigma(\Theta(N)) \to \bb C$ and $j\in \{1,\dots,m\}$ the 
    operator $\Xi_j\left( \int_{\sigma(\Theta_j(N))} f \, dE_j \right)$ belongs to $\{N,N^+\}''$.
\end{lemma}
\begin{proof}
    Take $C\in \{N,N^+\}'=\{A,B\}' \subseteq \bigcap_{j=1,\dots,m} (T_jT_j^+)'$; see \eqref{zfv6lkc}.
    From \thref{Xidefeig} in \cite{KaPr2014} we conclude
    \[
	C \, \Xi_j\left( \int_{\sigma(\Theta_j(N))} f \, dE_j \right)
	= \Xi_j\left( \Theta_j(C) (\int_{\sigma(\Theta_j(N))} f \, dE_j) \right) \,.
    \]
    Since $\Theta_j$ is a homomorphism, $\Theta_j(C)$ commutes with $\Theta_j(N)$ and, in turn, with 
    $\int_{\sigma(\Theta_j(N))} f \, dE_j$. Hence, employing \thref{Xidefeig} in \cite{KaPr2014} once 
    more, the above expression coincides with
    \[
	\Xi_j\left( (\int_{\sigma(\Theta_j(N))} f \, dE_j) \, \Theta_j(C) \right) 
	= \Xi_j\left( \int_{\sigma(\Theta_j(N))} f \, dE_j \right) C \,.
    \]
\end{proof}

\begin{definition}\thlab{Psidef}
Denoting by $\mf B\big(\sigma(\Theta(N))\big)$ the $*$-algebra of complex valued, 
bounded and measurable functions on $\sigma(\Theta(N))$, 
for $(r,f_1,\dots,f_m) \in \mc R:=\bb C[x,y]\times \mf B\big(\sigma(\Theta(N))\big) 
\times \dots \times \mf B\big(\sigma(\Theta(N))\big)$ we set
\[
    \Psi(r,f_1,\dots,f_m) := r(A,B) + \sum_{k=1}^m \Xi_k\left( \int_{\sigma(\Theta_k(N))} f_k \, dE_k \right) \,.
\]
By $\mc N$ we denote the set of all $(r,f_1,\dots,f_m) \in \mc R$ such that
\[
    r +\sum_{k=1}^m f_k p_k = 0 \ \ \text{ on } \ \ \sigma(\Theta(N))\setminus V_{\bb R}(I)
\]    
and such that there exist $u_1,\dots,u_m \in \bb C[x,y]$ with \ $r=\sum_{k=1}^m u_k p_k$ \ and
\[
  (f_j + u_j)(z)=0 \ \text{ for } \ j=1,\dots,m, \ z\in V_{\bb R}(I)\cap \sigma(\Theta(N)) \,. 
\]  
\end{definition}

\begin{remark}\thlab{compmist}
    Obviously, $\Psi$ is linear. From $\Xi_j(D^*)=\Xi_j(D)^+$ we easily deduce
    $\Psi(r^\#,\overline{f_1},\dots,\overline{f_m})=\Psi(r,f_1,\dots,f_m)^*$.
    Moreover, $\mc N$ constitutes a linear subspace of $\mc R$ invariant under
    $.^\#: (r,f_1,\dots,f_m)\mapsto (r^\#,\overline{f_1},\dots,\overline{f_m})$.
\end{remark}

\begin{lemma}\thlab{lre3}
      If $(r,f_1,\dots,f_m) \in \mc N$, then $\Psi(r,f_1,\dots,f_m)=0$.
\end{lemma}
\begin{proof}
    Due to \eqref{tt2479} $r=\sum_{k=1}^m u_k p_k$ implies
    \[
      r(A,B) = \sum_{k=1}^m p_k(A,B)u_k(A,B)  
	  = \sum_{k=1}^m  \Xi_k\big( u_k(\Theta_k(A),\Theta_k(B))\big) \,. 
    \]
    From this and \thref{lre0} we obtain
    \[
	\Psi(r,f_1,\dots,f_m) = \sum_{k=1}^m  \Xi_k\big( \int_{\sigma(\Theta_k(N))} (f_k + u_k) \, dE_k \big) =
    \]
    \[
	\Xi\left(\int\limits_{\sigma(\Theta(N))\setminus V_{\bb R}(I)} 
	      \sum_{k=1}^m \frac{f_kp_k + u_kp_k}{\sum_{l=1}^m p_l} \, dE 
	    + \sum_{k=1}^m R_kR_k^* \hspace{-2mm} \int\limits_{\sigma(\Theta(N))\cap V_{\bb R}(I)} 
	    \hspace{-2mm} (f_k + u_k) \, dE \right) = 0 \,.
    \]
\end{proof}

\begin{lemma}\thlab{lre4}
    For $(r,f_1,\dots,f_m), (s,g_1,\dots,g_m) \in \mc R$ have
    \begin{align*}
	\Psi(r,f_1,\dots, & f_m) \, \Psi(s,g_1,\dots,g_m) =
	    \\ & = \Psi(rs, rg_1+sf_1 + f_1 \sum_{k=1}^m g_k p_k,\dots, rg_m+sf_m + f_m \sum_{k=1}^m g_kp_k )
	    \\ & = \Psi(rs, rg_1+sf_1 + g_1 \sum_{k=1}^m f_k p_k,\dots, rg_m+sf_m + g_m \sum_{k=1}^m f_kp_k ) \,.
    \end{align*}
\end{lemma}
\begin{proof}
    By \thref{lre1} we have
    \[
	\Psi(r,f_1,\dots,f_m)\Psi(s,g_1,\dots,g_m) = r(A,B)s(A,B) 
    \]
    \[
	+ \sum_{k=1}^m r(A,B)\Xi_k( \int_{\sigma(\Theta_k(N))} g_k \, dE_k ) +
	\sum_{j=1}^m \Xi_j( \int_{\sigma(\Theta_j(N))} f_j \, dE_j ) s(A,B) 
    \]
    \[
	+
	\sum_{j,k=1}^m \Xi_j\left( \int_{\sigma(\Theta_j(N))} f_j \, dE_j \right)
	      \Xi_k\left(\int_{\sigma(\Theta_k(N))} g_k \, dE_k\right) 
    \]
    \[
	= (rs)(A,B) + \sum_{k=1}^m \Xi_k( \int_{\sigma(\Theta_k(N))} r g_k \, dE_k ) +
	\sum_{j=1}^m \Xi_j( \int_{\sigma(\Theta_j(N))} s f_j \, dE_j )  
    \]
    \[
	+
	\sum_{j=1}^m \Xi_j\left( \sum_{k=1}^m \int_{\sigma(\Theta_j(N))} f_jg_k \, p_k \, dE_j \right)  \,,
    \]
    where this last term can also be written as
    \[
	\sum_{j=1}^m \Xi_j\left( \sum_{k=1}^m \int_{\sigma(\Theta_j(N))} f_k g_j \, p_k \, dE_j \right) \,.
    \]
\end{proof}

\begin{definition}\thlab{Psidefpost}
We provide $\mc R$ with a multiplication in the following way:
\begin{multline*}
    (r,f_1,\dots,f_m) \cdot (s,g_1,\dots,g_m) := 
    \\
    (rs, rg_1+sf_1 + f_1 \sum_{j=1}^m g_j p_j,\dots, rg_m+sf_m + f_m \sum_{j=1}^m g_jp_j ) \,.
\end{multline*}
\end{definition}

\begin{remark}\thlab{compmistpost}
    Obviously, $\cdot$ is bilinear and compatible with $.^\#$ as defined in \thref{compmist}. 
    It is elementary to check its associativity.
    
    Moreover, for $(r,f_1,\dots,f_m) \in \mc N$ and $(s,g_1,\dots,g_m)\in \mc R$ we have
    $rs +\sum_{j=1}^m p_j (rg_j + sf_j + f_j \sum_{k=1}^m g_k p_k) = 
    (r +\sum_{j=1}^m f_j p_j) (s +\sum_{k=1}^m g_k p_k) = 0$ on $\bb C \setminus V_{\bb R}(I)$.
    For the corresponding $u_1,\dots,u_m \in \bb C[x,y]$ with
    $r=\sum_{j=1}^m u_j p_j$ and $(f_j + u_j)(z)=0$ for all $z\in V_{\bb R}(I)$ we have
    $rs = \sum_{j=1}^m (u_j s) p_j$ and
    \[
	rg_j+sf_j + f_j \sum_{k=1}^m g_k p_k + u_j s = r g_j + f_j \sum_{k=1}^m g_k p_k = 0
    \]
    on $V_{\bb R}(I)$ since $r$ and the $p_j$ vanish there. 
    Hence, $\mc N$ is a right ideal. Similarly, one shows that it is also a left ideal.
    Finally, the commutator
    \[
        (r,f_1,\dots,f_m) \cdot (s,g_1,\dots,g_m) -  (s,g_1,\dots,g_m) \cdot (r,f_1,\dots,f_m) =
    \]
    \[
        (0,\sum_{j=1}^m (f_1g_j - g_1f_j) p_j, \dots, \sum_{j=1}^m (f_mg_j - g_mf_j) p_j)
    \]
    belongs to $\mc N$. Consequently, $\mc R/\mc N$ is a commutative $*$-algebra.
\end{remark}

Gathering the previous results we obtain the final result of the present section.

\begin{theorem}\thlab{homo1}
  $\Psi/\mc N: (r,f_1,\dots,f_m) + \mc N \mapsto \Psi(r,f_1,\dots,f_m)$ 
  is a well-defined $*$-homomorphism from $\mc R/\mc N$ into $\{N,N^+\}'' \subseteq B(\mc K)$.
\end{theorem}

\section{Algebra of Zero-dimensional Ideals}

By the Noether-Lasker Theorem (see for example \cite{CLO1}, 
Theorem 7, Chapter 4, \S7) any ideal $I$ in $\bb C[x,y]$ admits a minimal primary decomposition
\begin{equation}\label{primdec}
    I = Q_1\cap \dots \cap Q_l \,.
\end{equation}
$Q_j$ being a primary ideal means that $fg\in Q_j$ implies $f\in Q_j$ or $g^k \in Q$ for some $k\in \bb N$, 
and minimal means that $Q_j \not\supseteq \bigcap_{i\neq j} Q_i$ for all $j=1,\dots,l$ and 
$P_j \neq P_i$ for $i\neq j$, where $P_j$ denotes the radical
\[
    \sqrt{Q_j} := \{f\in \bb C[x,y]: f^k \in Q_j \ \text{ for some } \ k \in \bb N\} \,.
\]
For an ideal $I$ in $\bb C[x,y]$ such a decomposition is in general not unique. Nevertheless,
the First Uniqueness Theorem on minimal primary decompositions states that
the number $l\in \bb N$ and the radicals $P_1,\dots,P_l$ are uniquely determined by $I$; see for example
\cite{BW}, Theorem 8.55 on page 362.
Moreover, the Second Uniqueness Theorem on minimal primary decompositions says that
if $Q_1'\cap \dots \cap Q_l'=I=Q_1\cap \dots \cap Q_l$ are minimal primary decompositions ordered such that
$P_j=\sqrt{Q_j} = \sqrt{Q_j'}$ for $j=1,\dots,l$ and if $P_k$ is minimal in $\{P_1,\dots,P_l\}$ 
with respect to $\subseteq$, then $Q_k'=Q_k$; see for example
\cite{BW}, Theorem 8.56 on page 364.

Assume now that $I$ is a zero-dimensional ideal in $\bb C[x,y]$, i.e.\ 
\[
    \dim \bb C[x,y] / I < \infty \,. 
\]
For necessary and sufficient conditions see for example \cite{BW}, Theorem 6.54 and 
Corollary 6.56 on pages 274 and 275 and \cite{CLO2}, page 39 and 40. 
Let \eqref{primdec} be a minimal primary decomposition. 
Then any $Q_j$, 
and in turn $P_j \supseteq Q_j$, is also zero-dimensional. 
In particular, $\bb C[x,y]/P_j$ is a finite integral domain, and hence, a field.
In turn, the radicals $P_1,\dots,P_l$ of $Q_1,\dots,Q_l$ are maximal ideals.
By \cite{CLO1}, Theorem 11, Chapter 4, \S5, this means that the $P_j$ are generated by $x-a_{x,j}, y-a_{y,j}$, i.e.\ 
$P_j=\langle x-a_{x,j}, y-a_{y,j} \rangle$, for pairwise distinct $a_j=(a_{x,j},a_{y,j}) \in \bb C^2$.
Consequently, any $P_k$ is minimal in $\{P_1,\dots,P_l\}$, and by what was said above, 
\eqref{primdec} is the unique minimal primary decomposition of $I$.

By Hilbert's Nullstellensatz (see for example \cite{CLO1}, Theorem 2, Chapter 4, \S1) the set $V(Q_j)$
of common zeros in $\bb C^2$ of all $f\in Q_j$ coincides with $V(P_j)=\{a_j\}$. 
By \cite{CLO1}, Theorem 7, Chapter 4, \S3, we also have 
\[
    V(I)= V(Q_1) \cup \dots \cup V(Q_l) = \{a_1,\dots,a_l\} \,.
\]
Since 
$V(Q_j + Q_i) = V(Q_j)\cap V(Q_i) = \{a_j\}\cap \{a_i\}=\emptyset$ (see \cite{CLO1}, Theorem 4, Chapter 4, \S3)
for $i\neq j$, the weak Nullstellensatz (see for example \cite{CLO1}, Theorem 1, Chapter 4, \S1) yields
$Q_j + Q_i=\bb C[x,y]$. Hence, by the Chinese Remainder Theorem the mapping
\begin{equation}\label{chreth}
    \phi: \ \left\{
	\begin{array}{lcr}
		 \bb C[x,y]/ I & \to & ({\bb C}[x,y]/ Q_1)\times \dots \times ({\bb C}[x,y]/ Q_l) \,, \\
		x+I & \mapsto & (x+Q_1,\dots,x+Q_l) 
	\end{array}\right.
\end{equation}
constitutes an isomorphism, and $I=\prod_{j=1}^l Q_j$.

%
%
%

\begin{remark}\thlab{qbem}\hspace*{0mm}
    \begin{enumerate}
    \item Since the ring ${\bb C}[x,y]/ Q_j$ is finite dimensional, its invertible elements $f+Q_j$ are exactly
    those, for which $fg \in Q_j$ implies $g\in Q_j$. $Q_j$ being primary this is equivalent to 
    $f\not\in P_j$. Hence, $f+Q_j$ is invertible in ${\bb C}[x,y]/ Q_j$ if and only if $f(a_j)\neq 0$.
    \item
    As $\sqrt{Q_j}=P_j$ we have $(x-a_{x,j})^m,(y-a_{y,j})^n\in Q_j$ for sufficiently large 
    $m, n \in \bb N$. Therefore, the ideal $P_j\cdot Q_j$ contains $(x-a_{x,j})^{m+1},(y-a_{y,j})^{n+1}$.
    Thus, $P_j\cdot Q_j$ is also zero-dimensional and $\sqrt{P_j\cdot Q_j}=P_j$.
    \end{enumerate}
\end{remark}

\begin{definition}
    For $a\in V(I)$ we set by $Q(a):=Q_j$ and $P(a):=P_j$, 
    where $j$ is such that $a=a_j$. By $d_x(a)$ ($d_y(a)$)
    we denote the smallest natural number $m$ ($n$) such that $(x-a_{x})^{m} \in Q(a)$
    ($(y-a_{y})^{n} \in Q(a)$).
    Moreover, for $a\in V(I)$ we set
    \[
	\mc A(a) := {\bb C}[x,y]/ (P(a)\cdot Q(a)) \ \ \text{ and } \ \ \mc B(a) := {\bb C}[x,y]/ Q(a) \,.
    \] 
\end{definition}

Since $P(a)\cdot Q(a)$ and $Q(a)$ are ideals with finite codimension satisfying $P(a)\cdot Q(a) \subseteq Q(a)$,
$\mc A(a)$ and $\mc B(a)$ are finite dimensional algebras with $\dim \mc A(a) \geq \dim \mc B(a)$.

\begin{remark}\thlab{sternop1}
    Assume that $I$ is invariant under $.^\#$, where $f^\#(x,y):=\overline{f(\bar x,\bar y)}$. 
    This is for sure the case if $I$ is generated by real polynomial $p_1,\dots,p_m$.
    Then $V(I) \subseteq\bb C^2$ is invariant under $(z,w) \mapsto (z,w)^\#:=(\bar z,\bar w)$.
    Moreover, it is elementary to check that with $Q$ also $Q^\#$ is a primary ideal. Hence, with
    $I=Q_1\cap \dots \cap Q_l$ also $I=I^\#=Q_1^\#\cap \dots \cap Q_l^\#$ is a minimal primary decomposition.
    By the uniqueness of the minimal primary decomposition for our zero dimensional ideal $I$ one has
    $Q(a)^\# = Q(a^\#)$ for all $a\in V(I)$.
    
    Consequently, $f\mapsto f^\#$ induces a conjugate linear bijection from $\mc A(a)$ ($\mc B(a)$)
    onto $\mc A(a^\#)$ ($\mc B(a^\#)$).
\end{remark}

For the following note that
if we conversely start with primary and zero-dimensional ideals $Q_1,\dots,Q_l$ with 
$\sqrt{Q_i}\neq \sqrt{Q_j}$ for $i\neq j$, then
$I:=Q_1\cap \dots \cap Q_l$ is also zero-dimensional, and by the above mentioned uniqueness statement, 
$Q_1\cap \dots \cap Q_l$ is indeed the unique minimal primary decomposition of $I$.

\begin{proposition}\thlab{algle1}
    Let $I$ be a zero-dimensional ideal in ${\bb C}[x,y]$ which is generated by $p_1,\dots,p_m$, and 
    let $I=\bigcap_{a\in V(I)} Q(a)$ be its unique primary decomposition. Assume that $W$ is a subset of $V(I)$. 
    Then 
    \[
	J:=\bigcap_{a\in V(I)\setminus W} Q(a) \cap \bigcap_{a\in W} (P(a)\cdot Q(a))
    \]
    is also a zero-dimensional ideal satisfying $J\subseteq I$. The mapping
    \[
	\psi: \ \left\{
	\begin{array}{lcr}
		 \bb C[x,y]/ J & \to & \varprod\limits_{a\in V(I)\setminus W} \big({\bb C}[x,y]/ Q(a)\big) \times 
					\varprod\limits_{a\in W} \big({\bb C}[x,y]/ (P(a)\cdot Q(a))\big)	\,, \\
		x+I & \mapsto & ((x+Q(a))_{a\in V(I)\setminus W},(x+(P(a)\cdot Q(a)))_{a\in W}) 
	\end{array}\right.
    \]
    is an isomorphism, and any $p\in J$ can be written in the form $p=\sum_{j} u_j p_j$, where 
    $u_j(a)=0$ for all $a\in W$.
\end{proposition}
\begin{proof}
    We already mentioned that $P(a)\cdot Q(a)$ is zero-dimensional with 
    $\sqrt{P(a)\cdot Q(a)}=P(a)$ and that the intersection 
    $J = \bigcap_{a\in V(I)\setminus W} Q(a) \cap \bigcap_{a\in W} P(a)\cdot Q(a)$ is the unique primary 
    decomposition of the zero-dimensional $J$. The isomorphism property of $\psi$ is a special case 
    of the corresponding fact concerning $\phi$; see \eqref{chreth}. We also have 
    \begin{align*}
	J & =\prod_{a\in V(I)\setminus W} Q(a) \cdot \prod_{a\in W} P(a)\cdot Q(a) = 
	\prod_{a\in V(I)} Q(a) \cdot \prod_{a\in W} P(a) \\ & = I \cdot \prod_{a\in W} P(a) = 
	\Big\langle p_1\cdot \prod_{a\in W} P(a), \dots,  p_m \cdot \prod_{a\in W} P(a)\Big\rangle \,. 
    \end{align*}
    This means that any $p\in J$ has a representation $p=\sum_{j} u_j p_j$ with 
    $u_j\in \prod_{a\in W} P(a)=\bigcap_{a\in W} P(a)$. Hence, $u_j(a)=0$ for all $a\in W$.
\end{proof}

\begin{example}\thlab{vorteil1}
    Assume that $I$ is generated by two polynomial $p_1, p_2 \in {\bb C}[x,y]$ such that
    $p_1$ only depend on $x$ and $p_2$ only depends on $y$. The set $V(I)$ of common zeros of
    $I$, or equivalently of $p_1$ and $p_2$, in $\bb C^2$ then consists of all points of the form
    $(z,w)$, where $z\in \bb C$ is a zero of $p_1$ and $w\in \bb C$ is a zero of $p_2$,
    i.e.\ $V(I)=p_1^{-1}\{0\} \times p_2^{-1}\{0\}$. For $z\in p_1^{-1}\{0\}$ denote by 
    $\mf d_1(z)$ $p_1$'s degrees of the zero at $z$, and for $w\in p_2^{-1}\{0\}$ denote by 
    $\mf d_2(w)$ $p_2$'s degrees of the zero at $w$.
        
    Given $p(x,y) \in {\bb C}[x,y]$ we can apply polynomial division in one variable twice, 
    once with respect to $x$ and once $y$, on order to see that
    \[
	p(x,y) = p_1(x)\cdot u(x,y) + p_2(y)\cdot v(x,y) + q(x,y) 
    \] 
    with $u(x,y), v(x,y), q(x,y) \in {\bb C}[x,y]$ such that the 
    degree of $q(x,y)$, seen as a polynomial on $x$, is less then the degree of $p_1$, 
    and such that the 
    degree of $q(x,y)$, seen as a polynomial on $y$, is less then the degree of $p_2$;
    see \thref{einbett2pre} in \cite{Ka2015}. Hence, $I$ is zero-dimensional.
    Moreover, writing $p_1(x)$ and $p_2(y)$ as products of linear factors, 
    it follows that $p\in I$ if and only if
    \begin{equation}\label{jjww306}
	p \in \langle (x-z)^{\mf d_1(z)},(y-w)^{\mf d_2(w)} \rangle :=Q((z,w)) \,, 
    \end{equation}
    for all $z\in {p_1}^{-1}\{0\}, w\in {p_2}^{-1}\{0\}$.
    Since $Q((z,w))$ is a primary ideal in ${\bb C}[x,y]$, 
    \[
	I = \bigcap_{(z,w)\in p_1^{-1}\{0\} \times p_2^{-1}\{0\}} Q((z,w))
    \]
    is the minimal primary decomposition of $I$. For the respective radicals we have
    $P((z,w))=\langle x-z, y-w \rangle$. Moreover, $P((z,w))\cdot Q((z,w))$
    coincides with
    \[
	\langle (x-z)^{\mf d_1(z)+1}, (x-z)^{\mf d_1(z)}(y-w),
					  (x-z)(y-w)^{\mf d_2(w)} ,(y-w)^{\mf d_2(w)+1} \rangle \,.
    \]
    Therefore, $\mc A((z,w))={\bb C}[x,y]/(P((z,w))\cdot Q((z,w)))$ 
    ($\mc B((z,w))={\bb C}[x,y]/Q((z,w))$) is isomorphic to
    $\mc A_{\mf d_1(z),\mf d_2(w)}$ ($\mc B_{\mf d_1(z),\mf d_2(w)}$) as introduced in
    \thref{muldefb1}, \cite{Ka2015}.
\end{example}

\section{Function classes}

In the present section we make the same assumptions and use the same
notation as in Section \ref{abstrfunccal}. In addition, we assume that
the ideal $I$ generated by all real definitizing polynomials
is zero-dimensional. We fix real, definitizing 
polynomials $p_1,\dots,p_m$ which generate $I$.
For the zero-dimensional $I$ we apply the same 
notation as in the previous section. 

The variety $V(I)=\{a_1,\dots,a_l\} \subseteq \bb C^2$ of common zeros of all $f\in I$
will be split up as
\[
    V(I) = \underbrace{(V(I)\cap \bb R^2)}_{=V_{\bb R}(I)} \, \dot\cup \, (V(I)\setminus\bb R^2) \,,
\]
where we consider $V_{\bb R}(I)$ as a subset of $\bb C$; see \eqref{nullstmereal}.

\begin{definition}
    By $\mc M_{N}$ we denote the set of functions $\phi$ defined on 
    \[
      \underbrace{\big(\sigma(\Theta(N)) \cup V_{\bb R}(I) \big)}_{\subseteq \bb C} \, \dot\cup \, 
      \underbrace{(V(I)\setminus \bb R^2)}_{\subseteq \bb C^2}
    \]
    with $\phi(z) \in \bb C$ for $z\in \sigma(\Theta(N))\setminus V_{\bb R}(I)$,
    $\phi(z) \in \mc A(z)$ for $z \in V_{\bb R}(I)$,
    $\phi(z) \in \mc B(z)$ for $z \in V(I)\setminus \bb R^2$.

    We provide $\mc M_{N}$ pointwise with scalar multiplication, addition and multiplication.
    We also define a conjugate linear involution $.^\#$ on $\mc M_{N}$ by 
\begin{align*}
	& \phi^\#(z) := \overline{\phi(z)} \ \ \text{ for } \ \ z\in \sigma(\Theta(N))\setminus V_{\bb R}(I), \\
	& \phi^\#(z) := \phi(z)^\# \ \ \text{ for } \ \ z \in V_{\bb R}(I)  \\
	& \phi^\#(\xi,\eta) := \phi(\bar \xi,\bar \eta)^\# \ \ \text{ for } \ \ (\xi,\eta) \in V(I)\setminus \bb R^2 \,.
\end{align*}
    
\end{definition}

With the operations introduced above $\mc M_{N}$ is a commutative $*$-algebra as can be verified 
in a straight forward manner; see \thref{sternop1}.

\begin{definition}\thlab{feinbetefab}
    Let $f: \dom f \to \bb C$ be a function with $\dom f \subseteq \bb C^2$ such that
    $\tau\big(\sigma(\Theta(N)) \cup V_{\bb R}(I)\big) \subseteq \dom f$, where 
    $\tau : \bb C \to \bb C^2, \ (x+iy)\mapsto (x,y)$, such that
    $f\circ \tau$ is sufficiently smooth -- more exactly, at least 
    $d_x(z) + d_y(z) - 1$ times continuously differentiable -- 
    on a sufficiently small open neighbourhood $z$ for each $z\in V_{\bb R}(I)$, and
    such that $f$ is holomorphic on an open neighbourhood of $V(I)\setminus \bb R^2 \ (\subseteq \bb C^2)$.

    Then $f$ can be considered as an element $f_{N}$ of $\mc M_{N}$ by setting
    $f_{N}(z) := f\circ \tau(z)$ for $z\in \sigma(\Theta(N)) \setminus V_{\bb R}(I)$, by
\begin{multline*}
    f_{N}(z) := \sum_{(k,l)\in J(z)}
	\frac{1}{k!l!} \, \frac{\partial^{k+l}}{\partial a^k\partial b^l} 
	f\circ \tau(a+ib)\vert_{a+ib=z} \cdot \\ \cdot (x-\Re z)^k (y-\Im z)^l  + (P(z)\cdot Q(z)) \in \mc A(z)  
\end{multline*}
    for $z\in V_{\bb R}(I)$, where 
    \[
	J(z)= (\{0,\dots,d_x(z)-1\}\times \{0,\dots,d_y(z)-1\}) \cup \{(d_x(z),0),(0,d_y(z))\} \,, 
    \]
    and by 
\begin{multline*}
	f_{N}(\xi,\eta) := \sum_{k=0}^{d_x(\xi,\eta)-1} \sum_{l=0}^{d_y(\xi,\eta)-1} 
		\frac{1}{k!l!} \, \frac{\partial^{k+l}}{\partial z^k\partial w^l} 
		f(z,w)\vert_{(z,w)=(\xi,\eta)} \cdot \\ \cdot (x-\xi)^k (y-\eta)^l + Q((\xi,\eta)) \in \mc B((\xi,\eta))  \,,
\end{multline*}
    for $(\xi,\eta)\in V(I)\setminus \bb R^2$.
\end{definition}

\begin{remark}\thlab{bweuh30}
By the Leibniz rule $f\mapsto f_{N}$ is compatible with multiplication. Obviously, it is also compatible with 
addition and scalar multiplication. If we define for a function $f$ as in \thref{feinbetefab} the function 
$f^\#$ by $f^\#(z,w) = \overline{f(\bar z,\bar w)}, \ (z,w) \in \dom f$, then we also have 
$(f^\#)_{N} = (f_{N})^\#$. 
\end{remark}

\begin{remark}\thlab{bweuh30po}
    A special type of functions $f$ as in \thref{feinbetefab} are polynomials in
    two variables, i.e.\ $f\in \bb C[x,y]$. Since for $z\in V_{\bb R}(I)$ and $(k,l)\not\in J(z)$
    we have $(x-\Re z)^k (y-\Im z)^l \in P(z)\cdot Q(z)$,
    \[
	f_N(z)=f + (P(z)\cdot Q(z)) \in \mc A(z) \,.
    \]
    Similarly, $f_{N}(\xi,\eta) = f + Q((\xi,\eta))\in \mc B((\xi,\eta))$ for $(\xi,\eta)\in V(I)\setminus \bb R^2$.
    
    In particular, for $f=\mathds{1}$ the element $f_N(z)$
    is the multiplicative unite in $\mc A(z)$ or $\mc B(z)$ 
    for all $z\in \big(\sigma(\Theta(N)) \cup V_{\bb R}(I) \big)\dot\cup (V(I)\setminus \bb R^2)$.
\end{remark}

For the following recall for example from \cite{CLO1}, 
Theorem 4, Chapter 2, \S5, that any ideal in $\bb C[x,y]$
always has a finite number of generators.

\begin{definition}\thlab{abschaefu}
    For any $w \in \sigma(\Theta(N)) \cap V_{\bb R}(I)$ such that
    $w$ is not isolated in $\sigma(\Theta(N))$    
    let $h_1,\dots,h_n$ be generators of the ideal
    $Q(w)$.
    For a sufficiently small neighbourhood $U(w)$ of $w$ 
    let $\chi_w: U(w)\setminus \{w\} \to [0,+\infty)$ be 
    \[
	\chi_w(z):= \max_{j=1,\dots,n} |h_j(z)| \,, 
    \]
    where $h_j(z)$, as usually, stands for $h_j(\Re z,\Im z)$.
\end{definition}

Since $w$ is a common zero of all $h\in Q(w)$, we have $\chi_w(z) \to 0$ for $z\to w$.
Moreover, for any $h\in Q(w)$ the fact, that $h_1,\dots,h_n$ are generators of $Q(w)$, yields 
$h=O(\chi_w)$ as $z \to w$.

Moreover, if $\chi_w'$ is defined in a similar manner starting with 
generators $h_1',\dots,h_{n'}'$, then $\chi_w' = O(\chi_w)$ and $\chi_w = O(\chi_w')$ as $z \to w$.
Hence, as far as it concerns the order of growth towards $w$, the expression $\chi_w$ does not depend 
on the actually chosen generators.

%

\begin{definition}\thlab{FdefklM2}
    We denote by $\mc F_{N}$ the set of all elements 
    $\phi \in \mc M_{N}$ such that
    $z\mapsto \phi(z)$ is Borel measurable and bounded on 
    $\sigma(\Theta(N)) \setminus V_{\bb R}(I)$, and such that
    for each $w \in \sigma(\Theta(N)) \cap  V_{\bb R}(I)$, which
    is not isolated in $\sigma(\Theta(N))$,
    \begin{equation}\label{fn8qw3bab}
    \phi(z) - \phi(w)\vert_{x=\Re z, y = \Im z} = O(\chi_w(z)) \ \ \text{ as } \ \ 
	\sigma(\Theta(N)) \setminus V_{\bb R}(I) \ni z \to w \,.
    \end{equation}
\end{definition}

Note that in \eqref{fn8qw3bab} $\phi(w) \in \mc A(w)$ is a coset $p(x,y) + (P(w)\cdot Q(w))$ from 
${\bb C}[x,y]/ (P(w)\cdot Q(w))$, and $\phi(w)\vert_{x=\Re z, y = \Im z}$ stands for any representative
of this coset $\phi(w)$ considered as a function of $z$. In \eqref{fn8qw3bab} it does not matter what representative we take since
$q = O(\chi_w)$ as 
$z \to w$ for any $q\in Q(w)$, and hence, for any $q\in (P(w)\cdot Q(w))$.

\begin{remark}\thlab{vorteil2}
    Assume that our zero-dimensional ideal $I$ is generated by two
    definitizing polynomials $p_1 \in \bb R[x], p_2\in \bb R[y]$ as in \thref{vorteil1}.
    For $w\in V_{\bb R}(I)$, i.e.\ $(\Re w,\Im w)\in V(I)$, we conclude from 
    \eqref{jjww306} in \thref{vorteil1} that
    \[
	\chi_w(z):= \max(|(\Re z - \Re w)^{\mf d_1(\Re w)}|,|(\Im z - \Im w)^{\mf d_2(\Im w)}|) \,.
    \]
    Therefore, in this case the function class $\mc F_{N}$ here coincides exactly with
    the function class $\mc F_{N}$ introduced in \thref{FdefklM}, \cite{Ka2015}.
\end{remark}

\begin{example}\thlab{fedela33}
    For $(\xi,\eta) \in V(I)\setminus \bb R^2$ and $a\in \mc B((\xi,\eta))$ the function
    $a\delta_{(\xi,\eta)} \in \mc M_{N}$, which assumes the value $a$ at $(\xi,\eta)$ and the value zero on the rest of
    $\big(\sigma(\Theta(N)) \cup V_{\bb R}(I) \big)\dot\cup (V(I)\setminus \bb R^2)$,
    trivially belongs to $\mc F_{N}$.
    
    Correspondingly, $a\delta_w \in \mc F_{N}$ for a $w \in V_{\bb R}(I)$, which is an isolated point of 
    $\sigma(\Theta(N)) \cup V_{\bb R}(I)$, and for $a\in \mc A(w)$.
\end{example}

\begin{remark}\thlab{taylormehrdimremab}
    Let $h$ be defined on an open subset $D$ of $\bb R^2$ with values in $\bb C$.
    Moreover, assume that for given $m,n\in\bb N$ the function $h$ is 
    $m+n-1$ times continuously differentiable. Finally, fix $w\in D$.
    
    The well-known Taylor Approximation Theorem from multidimensional calculus then yields
    \[
	h(z) = \sum_{j=0}^{m+n-2} \sum_{\stackrel{k,l\in \bb N_0}{k+l=j}}  
		\frac{1}{k!l!}  
		\frac{\partial^j h}{\partial x^{k} \partial y^{l}}(w)\Re(z-w)^k \Im(z-w)^{l} 
		+ O(|z-w|^{m+n-1}) 
    \]
    for $z\to w$. Since 
    \begin{align*}
    |z-w|^{m+n-1} & \leq 2^{m+n-1}\max(|\Re(z-w)|^{m+n-1}, |\Im(z-w)|^{m+n-1}) 
    \\ & = O(\max(|\Re(z-w)|^m, |\Im(z-w)|^{n})) \,, 
    \end{align*}
    and since 
    $\Re(z-w)^k \Im(z-w)^{l} = O(\max(|\Re(z-w)|^m, |\Im(z-w)|^{n}))$ for 
    $k \geq m$ or $l\geq n$, we also have
\begin{multline*}
	h(z) = \sum_{k=0}^{m-1} \sum_{l=0}^{n-1} \frac{1}{k!l!}
		\frac{\partial^{k+l} h}{\partial x^{k} \partial y^{l}}(w)\Re(z-w)^k \Im(z-w)^{l} 
		\\ + O(\max(|\Re(z-w)|^m, |\Im(z-w)|^{n})) \,.
\end{multline*}
\end{remark}

\begin{lemma}\thlab{gehzuFab}
    Let $f: \dom f \ (\subseteq \bb C^2)\to \bb C$ be a function with the properties
    mentioned in \thref{feinbetefab}. Then $f_N$ belongs to $\mc F_{N}$.
\end{lemma}
\begin{proof}
    For a $w \in \sigma(\Theta(N)) \cap  V_{\bb R}(I)$, which
    is not isolated in $\sigma(\Theta(N))$, and $z\in \sigma(\Theta(N))\setminus V_{\bb R}(I)$
    sufficiently near at $w$
    by \thref{taylormehrdimremab} the expression
    \[
	 f_N(z) - f_N(w)\vert_{x=\Re z, y = \Im z} = 
    \]
    \[
	f(\Re z, \Im z) - \hspace{-2mm} \sum_{(k,l) \in J(w)}
	\frac{1}{k!l!} \, \frac{\partial^{k+l} f}{\partial x^k\partial y^l}(\Re w, \Im w) 
	\cdot (\Re z-\Re w)^k (\Im z-\Im w)^l 
    \]
    is a $O(\max(|\Re(z-w)|^{d_x(w)}, |\Im(z-w)|^{d_y(w)}))$, and therefore a $O(\chi_w(z))$
    as $z\to w$. Consequently $f_N\in \mc F_{N}$.
\end{proof}

\begin{lemma}\thlab{einduF33}
    If $\phi \in \mc F_{N}$ is such that $\phi(z)$ is invertible in $\bb C, \mc A(z)$ or $\mc B(z)$, respectively, 
    for all $z\in \big(\sigma(\Theta(N)) \cup V_{\bb R}(I) \big)\dot\cup (V(I)\setminus \bb R^2)$
    and such that 
    $0\in\bb C$ does not belong to the closure of 
    $\phi\big(\sigma(\Theta(N)) \setminus V_{\bb R}(I)\big)$, then
    $\phi^{-1}: z\mapsto \phi(z)^{-1}$ also belongs to $\mc F_{N}$.
\end{lemma}
\begin{proof}
    By the first assumption $\phi^{-1}$ is a well-defined object belonging to $\mc M_{N}$.
    Clearly, with $\phi$ also $z\mapsto \phi(z)^{-1}=\frac{1}{\phi(z)}$ is
    measurable on $\sigma(\Theta(N)) \setminus V_{\bb R}(I)$.
    By the second assumption of the present lemma 
    $z\mapsto \phi(z)^{-1}=\frac{1}{\phi(z)}$ is bounded on this set.
    
    It remains to verify \eqref{fn8qw3bab} for $\phi^{-1}$ at each
    $w \in \sigma(\Theta(N)) \cap  V_{\bb R}(I)$, which is not isolated in $\sigma(\Theta(N)$.
    To do so, first note that due to $\phi(w)$'s invertibility for $z\in \sigma(\Theta(N)) \setminus V_{\bb R}(I)$
    sufficiently near at $w$ we have $\phi(w)\vert_{x=\Re z, y = \Im z} = p(z)\neq 0$, 
    where $p(x,y)$ is a representative of $\phi(w)$.
    Now calculate   \\
    \begin{equation}\label{dfbz45pre}
	 \hspace*{-6cm}      \phi^{-1}(z) - \phi(w)^{-1}\vert_{x=\Re z, y = \Im z} =
    \end{equation}
    \begin{equation}\label{dfbz45}
	  \hspace*{-2cm} = \frac{1}{\phi(z)} - 
		\frac{1}{\phi(w)\vert_{x=\Re z, y = \Im z}} +
    \end{equation}
\begin{equation}\label{fght33}
	  \hspace*{4cm} + \frac{1}{\phi(w)\vert_{x=\Re z, y = \Im z}} - 
	      \phi(w)^{-1}\vert_{x=\Re z, y = \Im z} \,.
\end{equation}
    The expression in \eqref{dfbz45} can be written as
    \[
	\frac{1}{\phi(z) \cdot \phi(w)\vert_{x=\Re z, y = \Im z}} \cdot
		\left( \phi(z) - \phi(w)\vert_{x=\Re z, y = \Im z} \right) \,.
    \]
    Here $\frac{1}{\phi(z)}$ is bounded by assumption.
    The assumed invertibility of $\phi(w)$ implies the boundedness $\phi(w)\vert_{x=\Re z, y = \Im z}$
    on a certain neighbourhood of $w$. 
    From $\phi\in \mc F_N$ we then conclude that \eqref{dfbz45} is a 
    $O(\chi_w(z))$ for $z\to w$.

    \eqref{fght33} can be rewritten as
    \[
      - \frac{1}{\phi(w)\vert_{x=\Re z, y = \Im z}} \cdot
      \Big(\phi(w)\vert_{x=\Re z, y = \Im z} \cdot \phi(w)^{-1}\vert_{x=\Re z, y = \Im z}- 1\Big) \,.
    \]
    The product in the brackets is a representative of 
    $\phi(w)\cdot \phi(w)^{-1}= 1 + (P(w)\cdot Q(w)) \in \mc A(w)$. Hence, \eqref{fght33} equals to 
    $\frac{1}{\phi(w)\vert_{x=\Re z, y = \Im z}} q(\Re z,\Im z)$ for a $q\in (P(w)\cdot Q(w))$, and is 
    therefore a $O(\chi_w(z))$ for $z\to w$. Altogether \eqref{dfbz45pre} is a $O(\chi_w(z))$ for $z\to w$.
    Thus, $\phi^{-1}\in \mc F_N$.
\end{proof}

\section{Functional Calculus for zero-dimensional $I$}
\label{Funcalzero}

\begin{lemma}\thlab{uzgv24}
    For each $\phi\in\mc F_N$ there exists $p\in\bb C[x,y]$ and complex valued
    $f_1,\dots,f_m \in \mf B(\sigma(\Theta(N)) \cup V_{\bb R}(I))$ with 
    $f_j(z)=0$ for $z\in V_{\bb R}(I)$ such that
    \[
	\phi(z) = p_N(z) + \sum_j f_j(z) \, (p_j)_N(z) 
    \]
    for all $z\in \sigma(\Theta(N)) \cup V_{\bb R}(I)$, and that
    $\phi((\xi,\eta)) = p_N((\xi,\eta))$ for all $(\xi,\eta) \in V(I)\setminus \bb R^2$.
\end{lemma}
\begin{proof}
    We apply \thref{algle1} to $W=V_{\bb R}(I)$. The fact, that $\psi$ is an isomorphism, then 
    yields the existence of a polynomial $p\in \bb C[x,y]$ such that $p + (P(w)\cdot Q(w)) = \phi(w)$
    for all $w \in V_{\bb R}(I)$ and such that $p + Q((\xi,\eta)) =\phi((\xi,\eta))$
    for all $(\xi,\eta) \in V(I)\setminus \bb R^2$. 
    
    By \thref{bweuh30po} we have $\phi(w)= p + (P(w)\cdot Q(w)) =p_N(w) \in \mc A(w)$
    for $w\in V_{\bb R}(I)$.
    For $(\xi,\eta) \in V(I)\setminus \bb R^2$ we have $\phi((\xi,\eta)) = p + Q((\xi,\eta))= p_N((\xi,\eta)) \in 
    \mc B((\xi,\eta))$.
    
    For $j=1,\dots,m$ we set $f_j(z):=\frac{\phi(z) - p(z)}{\sum_k p_k(z)}$ if 
    $z\in \sigma(\Theta(N))\setminus V_{\bb R}(I)$
    (see \thref{speknorm}), and $f_j(z) = 0$ if 
    $z\in V_{\bb R}(I)$. On $\sigma(\Theta(N))\cup V_{\bb R}(I)$ we then have 
    \[
	\phi(z) = p_N(z) + \sum_j f_j(z) \, (p_j)_N(z) \,.
    \]
    It remains to verify that the functions $f_j$ are measurable and bounded on $\sigma(\Theta(N))\setminus V_{\bb R}(I)$. 
    The measurability easily follows
    from the definition of $f_j$ and the measurability of $\phi$ on this set. Since there are only finitely many
    points in $V_{\bb R}(I)$, the measurability of $f_j$ on $\sigma(\Theta(N))\cup V_{\bb R}(I)$ follows.
    
    Concerning boundedness, note that by \thref{gehzuFab} $\phi-p_N$ belongs to $\mc F_N$. Since any representative 
    $(\phi-p_N)(w)\vert_{x=\Re z, y = \Im z}$ of $(\phi-p_N)(w)\in \mc A(w)$ belongs to $P(w)\cdot Q(w) \subseteq Q(w)$, 
    we have
    $(\phi-p_N)(z)=O(\chi_w(z))$ as $z\to w$ for any $w\in \sigma(\Theta(N))\cap V_{\bb R}(I)$
    which is not isolated on $\sigma(\Theta(N))$.
    By \thref{speknorm} we have $\chi_w(z) = O(\sum_k p_k(z))$ as $z\to w$ for $z\in \sigma(\Theta(N))\setminus V_{\bb R}(I)$.
    Therefore,
    \[
	f_j(z) = \frac{\phi(z) - p(z)}{\sum_k p_k(z)} = O(1) \ \ \text{ as } \ \ z\to w
    \]
    for $z\in \sigma(\Theta(N))\setminus V_{\bb R}(I)$.
\end{proof}

\begin{definition}\thlab{deltadef}
Let $\Delta$ be the set of all pairs $(\phi; (p,f_1\vert_{\sigma(\Theta(N))},\dots,f_m\vert_{\sigma(\Theta(N))}))$
such that all assertions from \thref{uzgv24} hold true for $\phi$ and $ (p,f_1,\dots,f_m)$. 
\end{definition}

\begin{remark}\thlab{deltadefpo}
It is straight forward to check that
$\Delta$ is a linear subspace of 
$\mc F_N \times \Big(\bb C[x,y] \times \mf B\big(\sigma(\Theta(N))\big) \times \dots \times \mf B\big(\sigma(\Theta(N))\big)\Big)$, i.e.\
a linear relations. Moreover, it is easy to check that with 
$(\phi; (p,f_1\vert_{\sigma(\Theta(N))},\dots,f_m\vert_{\sigma(\Theta(N))}))$ also
$(\phi^\#; (p^\#, \overline{f_1\vert_{\sigma(\Theta(N))}},\dots,\overline{f_m\vert_{\sigma(\Theta(N))}}))$ belongs to $\Delta$; 
see \thref{compmist}.
\end{remark}

$\Delta$ is also compatible with multiplication as will be shown next.

\begin{lemma}\thlab{ujr7984}
    If both, $(\phi; (p,f_1\vert_{\sigma(\Theta(N))},\dots,f_m\vert_{\sigma(\Theta(N))}))$ and 
    $(\psi; (q,g_1\vert_{\sigma(\Theta(N))},\dots,g_m\vert_{\sigma(\Theta(N))}))$, belong to 
    $\Delta$, then also the pair $(\phi\cdot\psi; (r,h_1\vert_{\sigma(\Theta(N))},\dots,h_m\vert_{\sigma(\Theta(N))}))$
     belongs to $\Delta$, where (see \thref{Psidefpost})
     \begin{multline*}
	(r,h_1\vert_{\sigma(\Theta(N))},\dots,h_m\vert_{\sigma(\Theta(N))}) = \\ 
	(p,f_1\vert_{\sigma(\Theta(N))},\dots,f_m\vert_{\sigma(\Theta(N))}) \cdot 
	(q,g_1\vert_{\sigma(\Theta(N))},\dots,g_m\vert_{\sigma(\Theta(N))}) \,.
     \end{multline*}
\end{lemma}
\begin{proof}
      On $\sigma(\Theta(N)) \cup V_{\bb R}(I)$ we have 
      \[
	\phi(z) = p_N(z) + \sum_j f_j(z) (p_j)_N(z) \ \text{ and } \
	\psi(z) = q_N(z) + \sum_j g_j(z) (p_j)_N(z) \,.
      \] 
      Moreover,
      $f_j(z)=0=g_j$ for $z\in V_{\bb R}(I)$, and 
      $\phi((\xi,\eta)) = p_N((\xi,\eta))$, $\psi((\xi,\eta)) = q_N((\xi,\eta))$
      for all $(\xi,\eta) \in V(I)\setminus \bb R^2$.
      
      Since $p\mapsto p_N$ is compatible with multiplication, for $r=p\cdot q$ we have
      $(\phi\cdot\psi)((\xi,\eta)) = r_N((\xi,\eta))$ for all $(\xi,\eta) \in V(I)\setminus \bb R^2$.
      Clearly, $h_j = pg_j+qf_j + f_j \sum_{k=1}^m g_k p_k$ vanishes on $V_{\bb R}(I)$.
      For $z\in \sigma(\Theta(N)) \cup V_{\bb R}(I)$ we have
     \begin{multline*}
	  \phi(z) \, \psi(z) = p_N(z)\, q_N(z) + \\ 
	  \sum_j \Big(p_N(z) g_j(z) + q_N(z)f_j(z) + f_j(z) \sum_k g_k(z) (p_k)_N(z)\Big) \, (p_j)_N(z) \,, 
     \end{multline*}
      which, for $z\in V_{\bb R}(I)$, coincides with $r_N(z) = r_N(z) + \sum_j h_j(z) (p_j)_N(z)$.
      For $z\in \sigma(\Theta(N)) \setminus V_{\bb R}(I)$ the above equation can be written as 
      \begin{align*}
	  \phi(z)\, \psi(z) & = r(z) + 
	  \sum_j \Big(p(z) g_j(z) + q(z)f_j(z) + f_j(z) \sum_k g_k(z) p_k(z)\Big) \, p_j(z) \\ & =
	  r_N(z) + \sum_j h_j(z) \, (p_j)_N(z) \,.
      \end{align*}
\end{proof}

We are going to determine the multivalued part $\mul \Delta$ of $\Delta$.

\begin{lemma}\thlab{ujr43094}
    Assume that $p(x,y) \in\bb C[x,y]$ and 
    $f_1,\dots,f_m \in \mf B(\sigma(\Theta(N))\cup V_{\bb R}(I))$ with $f_j(z)=0$ for $z\in V_{\bb R}(I)$
    such that 
    \[
	0 = p_N(z) + \sum_j f_j(z) (p_j)_N(z) 
    \]
    on $\sigma(\Theta(N))\cup V_{\bb R}(I)$ and that $\phi((\xi,\eta)) = 0$ for all 
    $(\xi,\eta) \in V(I)\setminus \bb R^2$. Then $(p,f_1\vert_{\sigma(\Theta(N))},\dots,f_m\vert_{\sigma(\Theta(N))})$ belongs to 
    the ideal $\mc N$ in $\mc R$ as defined in \thref{Psidef}.  
\end{lemma}
\begin{proof}
    Clearly, $p +\sum_{j=1}^m f_j p_j = 0$ on $\sigma(\Theta(N))\setminus V_{\bb R}(I)$.

    According to \thref{bweuh30po} $p+(P(w)\cdot Q(w)) = 0 \in \mc A(w)$ for all $w\in V_{\bb R}(I)$ and
    $p+Q((\xi,\eta)) = 0 \in \mc B((\xi,\eta))$ for all 
    $(\xi,\eta) \in V(I)\setminus \bb R^2$. Hence, 
    $p\in \bigcap_{(\xi,\eta)\in V(I)\setminus \bb R^2} Q((\xi,\eta)) \cap \bigcap_{w\in V_{\bb R}(I)} (P(w)\cdot Q(w))$.
    By \thref{algle1} we therefore have $p = \sum_j u_jp_j$ with $u_j(w)=0$ for all $w\in V_{\bb R}(I)$.
    We see that $(f_j + u_j)(z)=0$ for all $z\in V_{\bb R}(I)\cap \sigma(\Theta(N))$. Thus, 
    $(p,f_1\vert_{\sigma(\Theta(N))},\dots,f_m\vert_{\sigma(\Theta(N))}) \in \mc N$. 
\end{proof}

Since by \thref{lre3} $\mul \Delta \subseteq \mc N \subseteq \ker \Psi$ the composition
$\Psi \Delta$ is a well-defined linear mapping from $\mc F_N$ into $B(\mc K)$.

\begin{definition}\thlab{calcueldef}
    For $\phi\in\mc F_N$ we set $\phi(N):=(\Psi \Delta)(\phi)$.
\end{definition}

By \thref{homo1}, \thref{ujr7984} and \thref{deltadefpo} the following result can be formulated.

\begin{theorem}\thlab{tu2989}
    $\phi\mapsto \phi(N)$ constitutes a $*$-homomorphism from $\mc F_N$ into $\{N,N^*\}'' \subseteq B(\mc K)$.
    It satisfies $p_N(N)=p(A,B)$ for all $p\in\bb C[x,y]$.
\end{theorem}
\begin{proof}
    The final assertion is clear because of $(p_N; (p,0,\dots,0)) \in \Delta$.
\end{proof}

\section{Spectral properties of the functional calculus}

For $w\in V_{\bb R}(I)$ we will need the following notation. By $\pi_w: \mc A(w) \to \mc B(w)$
we denote the mapping 
\[
    \pi_w(f+(P(w)\cdot Q(w))) = f+ Q(w) \,.
\]

\begin{lemma}\thlab{ehgtttpre}
    If $\phi\in \mc F_N$ vanishes everywhere except at a fixed $w\in V_{\bb R}(I)$ and 
    if $\pi_w \phi(w) = 0$, then
    \[
	\phi(N) = \Psi(0;g_1,\dots,g_m)
    \]
    for $g_1,\dots,g_m \in \mf B\big(\sigma(\Theta(N))\big)$ which vanish on 
    $(\sigma(\Theta(N)) \cup V_{\bb R}(I)) \setminus \{w\}$.
\end{lemma}
\begin{proof}
    Let $p(x,y)\in\bb C[x,y]$ and
    $f_1,\dots,f_m \in \mf B(\sigma(\Theta(N)) \cup V_{\bb R}(I))$ with 
    $f_j(z)=0$ for $z\in V_{\bb R}(I)$ such that
    \[
	\phi(z) = p_N(z) + \sum_j f_j(z) \, (p_j)_N(z) 
    \]
    for all $z\in \sigma(\Theta(N)) \cup V_{\bb R}(I)$, and that
    $p_N((\xi,\eta)) = \phi((\xi,\eta)) = 0$ for all $(\xi,\eta) \in V(I)\setminus \bb R^2$.
    The latter fact just means $p\in p((\xi,\eta)) \in Q((\xi,\eta))$.
    From $0=\phi(z) = p_N(z) + \sum_j f_j(z) \, (p_j)_N(z)$ for $z\in V_{\bb R}(I)\setminus \{w\}$ 
    we infer $p\in (P(z)\cdot Q(z))$. For $z=w$ this equation together with $\pi_w \phi(w) = 0$ yields
    $p \in Q(w)$.
    
    By \thref{algle1} $p=\sum_{j} u_j p_j$, where $u_j(z)=0$ for all $z\in V_{\bb R}(I) \setminus \{w\}$.
    We define $g_j$ to be zero on $\big(\sigma(\Theta(N)) \cup V_{\bb R}(I)\big)\setminus \{w\}$ and set $g_j(w)=u_j(w)$.
    The difference
    \[
	(p;f_1\vert_{\sigma(\Theta(N))},\dots,f_m\vert_{\sigma(\Theta(N))}) 
	- (0;g_1,\dots,g_m) = 
    \]
    \[
	(p; f_1\vert_{\sigma(\Theta(N))}-\delta_w(.) u_1(w),\dots,f_m\vert_{\sigma(\Theta(N))}-\delta_w(.) u_m(w))
    \]
    satisfies $p + \sum_j (f_j(z)-\delta_w(z) u_j(w)) p_j(z) = \phi(z) =0$ for $z\in \sigma(\Theta(N)) \setminus V_{\bb R}(I)$
    and $f_j(z)-\delta_w(z) u_j(w) + u_j(z) = 0$ for all $z\in V_{\bb R}(I)\cap \sigma(\Theta(N))$.
    It therefore belongs to the ideal $\mc N$ of $\mc R$. Consequently,
    \[
	\phi(N) = \Psi(p;f_1\vert_{\sigma(\Theta(N))},\dots,f_m\vert_{\sigma(\Theta(N))})
	= \Psi(0;g_1,\dots,g_m) \,.
    \]
\end{proof}

\begin{corollary}\thlab{ehgttt}
    Assume that $E\{w\} = 0$ for a fixed $w\in V_{\bb R}(I)$, which surely happens if $w\not\in \sigma(\Theta(N))$. 
    Then $\phi(N)=\psi(N)$ for all $\phi, \psi$ that coincide on 
    $\big((\sigma(\Theta(N)) \cup V_{\bb R}(I))\setminus \{w\} \big)\dot\cup (V(I)\setminus \bb R^2)$
    and that satisfy $\pi_w \phi(w) = \pi_w \psi(w)$. Here $\pi_w: \mc A(w) \to \mc B(w)$ is defined
    by $\pi_w(f+(P(w)\cdot Q(w))) = f + Q(w)$.
\end{corollary}
\begin{proof}
    By \thref{ehgtttpre} there exist $g_1,\dots,g_m \in \mf B\big(\sigma(\Theta(N))\big)$, which vanish on 
    $(\sigma(\Theta(N)) \cup V_{\bb R}(I)) \setminus \{w\}$, such that 
    \[
      \phi(N) - \psi(N) = \Psi(0;g_1,\dots,g_m) = \sum_{k=1}^m \Xi_k\left( \int_{\sigma(\Theta_k(N))} g_k \, dE_k \right)
    \]
    According to \thref{lre0} together with our assumption $E\{w\} = 0$,
    this operator vanishes.
\end{proof}

\begin{remark}\thlab{rieszproj33}
For $\zeta \in V(I)\setminus \bb R^2$ or a $\zeta \in V_{\bb R}(I)$, which is isolated in 
$\sigma(\Theta(N)) \cup V_{\bb R}(I)$, 
we saw in \thref{fedela33} that $a\delta_\zeta \in \mc F_{N}$.
If $a$ is the unite $e$ in $\mc B(\zeta)$ or in $\mc A(\zeta)$, i.e.\ the coset $1+Q(\zeta)$ for $\zeta\in V(I)\setminus \bb R^2$ 
or the coset $1+(P(\zeta)\cdot Q(\zeta))$ for $\zeta\in V_{\bb R}(I)$, then 
$(e\delta_\zeta)\cdot (e\delta_\zeta) = (e\delta_\zeta)$ together with the multiplicativity of $\phi\mapsto \phi(N)$
shows that $(e\delta_\zeta)(N)$ is a projection. It is a kind of Riesz projection corresponding to $\zeta$. 

We set $\xi:=\Re \zeta, \ \eta:=\Im \zeta$ if $\zeta \in V_{\bb R}(I)$ and
$(\xi,\eta):=\zeta$ if $\zeta \in V(I)\setminus \bb R^2$.
For $\lambda \in \bb C\setminus \{\xi+i\eta\}$ and for $s(z,w):=z+iw-\lambda$ we then have
$s_N \cdot (e\delta_\zeta) = \big(s_N(\zeta)\big)\delta_\zeta$.
As $s(\xi,\eta)\neq 0$, $s_N(\zeta)$ does not belong to $P(\zeta) \supseteq Q(\zeta)$.
Therefore, it is invertible in $\mc B(\zeta)$ or in $\mc A(\zeta)$. For its inverse $b$
we obtain
\[
    s_N \cdot (e\delta_\zeta) \cdot (b\delta_\zeta) = e\delta_\zeta \,.
\]
From $s_N(N) = N - \lambda$ we derive that $(N\vert_{\ran (e\delta_\zeta)(N)} - \lambda)^{-1}
=(b\delta_\zeta)(N)\vert_{\ran (e\delta_\zeta)(N)}$
on $\ran (e\delta_\zeta)(N)$. 
In particular, $\sigma(N\vert_{\ran (e\delta_\zeta)(N)}) \subseteq \{\xi+i\eta\}$.
\end{remark}

\begin{lemma}\thlab{deth5633}
    If $\phi\in \mc F_{N}$ vanishes on
    \[
      \big(\sigma(\Theta(N)) \cup (V_{\bb R}(I) \cap \sigma(N))\big) 
	  \dot\cup \{(\alpha,\beta)\in V(I)\setminus \bb R^2: \alpha+i\beta, \bar \alpha+i\bar \beta \in \sigma(N)\} \,, 
    \]
    then $\phi(N)=0$.
\end{lemma}
\begin{proof}
    Since any $w \in V_{\bb R}(I) \setminus \sigma(N)$ is isolated
    in $\sigma(\Theta(N)) \cup V_{\bb R}(I)$, we saw in \thref{rieszproj33} that
    for 
    \[
	\zeta \in \underbrace{\big(V_{\bb R}(I) \setminus \sigma(N)\big)}_{=:Z_1} \dot\cup
	\underbrace{\{(\alpha,\beta)\in V(I)\setminus \bb R^2: \alpha+i\beta \in \rho(N)\}}_{=:Z_2}
    \]
    the expression 
    $(e\delta_\zeta)(N)$ is a bounded projection commuting with $N$. Hence,
    $(e\delta_\zeta)(N)$ also commutes with $(N - (\xi+i\eta))^{-1}$, where
    $\xi:=\Re \zeta, \ \eta:=\Im \zeta$ if $\zeta \in Z_1$ and
    $(\xi,\eta):=\zeta$ if $\zeta \in Z_2$.
    
    Consequently, $N\vert_{\ran (e\delta_\zeta)(N)}- (\xi+i\eta)$ is invertible on $\ran (e\delta_\zeta)(N)$,
    i.e.\ $\xi+i\eta \not\in \sigma(N\vert_{\ran (e\delta_\zeta)(N)})$. In
    \thref{rieszproj33} we saw $\sigma(N\vert_{\ran (e\delta_\zeta)(N)}) \subseteq \{\xi+i\eta\}$.
    Hence, $\sigma(N\vert_{\ran (e\delta_\zeta)(N)}) = \emptyset$, which is impossible for
    $\ran (e\delta_\zeta)(N) \neq \{0\}$. Thus, $(e\delta_\zeta)(N) = 0$. 
    
    For $(\xi,\eta)\in Z_3:=\{(\alpha,\beta)\in V(I)\setminus \bb R^2: \bar\alpha+i\bar\beta \in \rho(N)\}$
    one has $(\bar \xi,\bar \eta)\in Z_2$. Hence,
    \[
	0=(e\delta_{(\bar\xi,\bar\eta)})(N)^*=(e^\#\delta_{(\xi,\eta)})(N)=(e\delta_{(\xi,\eta)})(N) \,.
    \]
    Since, by our assumption, $\phi$ is supported on $Z_1\cup Z_2 \cup Z_3$, we obtain 
    \[
	\phi(N) = (\hspace{-2mm}\sum_{\zeta \in Z_1\cup Z_2 \cup Z_3} \hspace{-2mm} 
	\phi(\zeta)\delta_\zeta \hspace{+2mm})(N) = 
	\sum_{\zeta \in Z_1\cup Z_2 \cup Z_3} \phi(\zeta) (e\delta_\zeta)(N) = 0 \,. 
    \]
 \end{proof}

As a consequence of \thref{deth5633} for $\phi \in \mc F_{N}$ the operator $\phi(N)$ only depends on
$\phi$'s values on
  \begin{multline}\label{eigmeng}
\big(\sigma(\Theta(N)) \cup (V_{\bb R}(I) \cap \sigma(N))\big) \dot\cup \\ \{(\alpha,\beta)\in V(I)\setminus \bb R^2: 
      \alpha+i\beta, \bar \alpha+i\bar \beta \in \sigma(N)\} \,.
  \end{multline}
Thus, we can, and will from now on, re-define the function class $\mc F_{N}$ for our functional calculus
so that the elements $\phi$ of $\mc F_{N}$ are functions on this set
with values in $\bb C, \mc A(z)$ or $\mc B(z)$,
such that $z\mapsto \phi(z)$ is measurable and bounded on $\sigma(\Theta(N)) \setminus V_{\bb R}(I)$
and such that \eqref{fn8qw3bab} holds true for every $w \in \sigma(\Theta(N)) \cap V_{\bb R}(I)$
which is not isolated in $\sigma(\Theta(N))$.

\begin{lemma}\thlab{wannboundinv33}
    If $\phi \in \mc F_{N}$ 
    is such that $\phi(z)$ is invertible in $\bb C, \mc A(z)$ or $\mc B(z)$, respectively, for all 
    $z$ in \eqref{eigmeng}, and such that 
    $0$ does not belong to the closure of 
    $\phi\big(\sigma(\Theta(N))\setminus V_{\bb R}(I) \big)$, then
    $\phi(N)$ is a boundedly invertible operator on $\mc K$ with $\phi^{-1}(N)$
    as its inverse.
\end{lemma}
\begin{proof}
    We think of $\phi$ as a function on $\big(\sigma(\Theta(N)) \cup V_{\bb R}(I) \big)\dot\cup (V(I)\setminus \bb R^2)$
    by setting $\phi(z)=e$ for all $z$ not belonging to \eqref{eigmeng}.
    Then all assumptions of \thref{einduF33} are satisfied. Hence $\phi^{-1} \in \mc F_{N}$, and we conclude from
    \thref{tu2989} and \thref{bweuh30po} that
    \[
	\phi^{-1}(N) \phi(N) = \phi(N) \phi^{-1}(N) = (\phi\cdot \phi^{-1})(N) = \mathds{1}_N(N) = I_{\mc K} \,. 
    \]
\end{proof}

\begin{corollary}\thlab{sigmaN}
    $\sigma(N)$ equals to
  \begin{multline}\label{specform}
    \sigma(\Theta(N)) \cup (V_{\bb R}(I) \cap \sigma(N))
    \cup \\ \{\alpha + i\beta :  (\alpha,\beta)\in V(I)\setminus \bb R^2, 
      \alpha+i\beta, \bar \alpha+i\bar \beta \in \sigma(N)\} \,.
  \end{multline}
In particular, $\sigma(N)\setminus \sigma(\Theta(N))$ is finite. 
\end{corollary}
\begin{proof}
    Since $\Theta$ is a homomorphism, we have $\sigma(\Theta(N))\subseteq \sigma(N)$.
    Hence, \eqref{specform} is contained in $\sigma(N)$.
    For the converse, consider the polynomial $s(z,w) = z+iw - \lambda$ for a $\lambda$
    not belonging to \eqref{specform}. We conclude that for any
    \[
      \zeta\in (V_{\bb R}(I) \cap \sigma(N)) \cup \{(\alpha,\beta)\in V(I)\setminus \bb R^2: 
      \alpha+i\beta, \bar \alpha+i\bar \beta \in \sigma(N)\} 
    \]
    the polynomial
    $s$ does not belong to $P(\zeta) \supseteq Q(\zeta)$. Hence, $s_N(\zeta)$ is invertible
    $\mc A(\zeta)$ or $\mc B(\zeta)$.
    Clearly, $s_N(\zeta)\neq 0$ for $\zeta\in \sigma(\Theta(N)) \setminus V_{\bb R}(I)$.
    Finally, $0$ does not belong to the closure of
    \[
	s_N\big(\sigma(\Theta(N))\setminus V_{\bb R}(I) \big) = s(\sigma(\Theta(N))\setminus V_{\bb R}(I))
	  \subseteq \sigma(\Theta(N)) - \lambda \,.
    \]
    Applying \thref{wannboundinv33}, we see that $s_N(N)=(N-\lambda)$ is invertible.
\end{proof}

\begin{remark}\thlab{fhwr3435}
    We set $K_r:=V_{\bb R}(I) \cap \sigma(N)$,
    \[
	Z:=\{(\alpha,\beta)\in V(I)\setminus \bb R^2: \alpha+i\beta, \bar \alpha+i\bar \beta \in \sigma(N)\} \,,
    \]
    and $K_i := \{\alpha + i\beta : (\alpha,\beta)\in Z\}$.
    Using \thref{sigmaN} we could re-define once more the functions $\phi\in\mc F_N$
    as functions $\phi$ on $\sigma(N)$ such that 
    \begin{enumerate}
    \item $\phi$ is complex valued, bounded and measurable on $\sigma(N)\setminus (K_r \cup K_i)$,
    \item $\phi(\zeta)\in \mc A(\zeta)$ for $\zeta \in K_r \setminus K_i$,
    \item $\phi(\zeta)\in \varprod_{(\alpha,\beta) \in Z, \alpha + i\beta = \zeta} \mc A(\zeta)$ for $\zeta \in K_i \setminus K_r$,
    \item $\phi(\zeta)\in \mc A(\zeta) \times \varprod_{(\alpha,\beta) \in Z, \alpha + i\beta = \zeta} \mc A(\zeta)$ for $\zeta \in K_r \cap K_i$;
    \item for a $w\in K_r$, which is not isolated in $\sigma(N)$, we have
    \[
	\phi(z) - p(\Re z,\Im z) = O(\chi_w(z)) \ \ \text{ as } \ \ 
	\sigma(N) \setminus (K_r \cup K_i) \ni z \to w \,,
    \]
    where $p$ is a representative of $\phi(w)$ for $w\in K_r \setminus K_i$ and 
    $p$ is a representative of the first entry of $\phi(w)$ for $w\in K_r \cap K_i$.
    \end{enumerate}
\end{remark}

\section{Special cases of definitizable operators}

Unitary and selfadjoint operators are special cases of
normal operators on Hilbert spaces as well as on Krein spaces. 
We will show how some well-known facts on definitizable selfadjoint
or unitary operators on a Krein space $\mc K$ can easily be obtain 
from the previously obtained results.

\subsection{Selfadjoint definitizable operators}

An operator $N\in B(\mc K)$ is by definition selfadjoint if $N=N^{+}$. Obviously,
$N\in B(\mc K)$ is selfadjoint if and only if $N$ is normal and satisfies
$p(A,B) = 0$, where $A=\frac{N+N^{+}}{2}, B=\frac{N-N^{+}}{2i}$ and
\[
    p(x,y) = y \in \bb R[x,y] \,.
\]
Therefore, according to \thref{definidef} any selfadjoint operator on a Krein space 
is definitizable normal, and the ideal $I$ generated by all real
definitizing polynomials contains $p(x,y)=y$.
Since the ideal generated by $p(x,y)=y$ is not zero-dimensional, 
the zero-dimensionality of $I$ implies the existence of at least one 
real definitizing polynomial of the form
\begin{equation}\label{ndydivba}
    y \cdot s(x,y) + t(x) \ \ \text{ with } \ \ s(x,y)\in \bb C[z,w], \ t(x)\in \bb C[x]\setminus \{0\} \,.
\end{equation}

\begin{proposition}
    The ideal $I$ is zero-dimensional if and only if 
    there exists a $t \in \bb R[x]\setminus \{0\}$ such that
    $[t(A)u,u] \geq 0, \ u \in \mc K$, i.e.\ $N=A$ is definitizable in 
    the classical sense; see \cite{KaPr2014}.
\end{proposition}
\begin{proof}
    Any $r(x,y) \in \bb C[z,w]$ can we written as
    $r(x,y) = y\cdot s_r(x,y) + t_r(x)$ with unique 
    $s_r(x,y)\in \bb C[z,w], \, t_r(x)\in \bb C[x]$.
    Hence, $r\in I$ if and only if $t_r(x) \in I$.
    The set of $I_x:=\{t_r : r\in I\}$ forms an ideal in $\bb C[x]$.
    If $I_x$ is the zero ideal, then $I=y\cdot\bb C[x,y]$ is not zero-dimensional.
    
    If $I_x\neq \{0\}$, then, applying the polynomial division, we see that 
    $\dim \bb C[x]/I_x <\infty$. This implies the zero-dimensionality of $I$.
    If $r(x,y)$ is a real definitizing polynomial as in \eqref{ndydivba}, then 
    \[
	[t(A)u,u] = [r(A,B)u,u] \geq 0, \ u \in \mc K \,,
    \]
    i.e.\ $t(x)$ is a definitizing polynomial.
\end{proof}

Assume that $N\in B(\mc K)$ is selfadjoint and that the ideal $I$ generated by all real
definitizing polynomials is zero-dimensional. Consequently, we can apply the functional 
calculus developed in Section \ref{Funcalzero}. 
From $p(x,y)=y\in I$ we conclude
\[
    a=(a_x,a_y) \in V(I) \Rightarrow a_y=p(a) = 0 \,.
\]
Hence, the elements of $V_{\bb R}(I)$ are contained in $\bb R$, and
$(\xi,\eta) \in V(I) \setminus \bb R^2$ yields $\eta=0$.
Moreover,
with $N$ also $\Theta(N)$ is selfadjoint in the Hilbert space
$\mc H$; see \thref{defNspaces} and \eqref{thetaVdef}. In particular,
$\sigma(\Theta(N)) \subseteq \bb R$. From \thref{sigmaN} we derive that
$\sigma(N)$ is contained in $\bb R$ up to finitely many points which are located 
in $\bb C \setminus \bb R$ symmetric with respect to $\bb R$.

\subsection{Unitary definitizable operators}

An operator $N\in B(\mc K)$ is by definition unitary if $N^{+} N = N N^{+} = I_{\mc K}$. Obviously,
$N\in B(\mc K)$ is unitary if and only if $N$ is normal and satisfies
$p(A,B) = 0$, where $A=\frac{N+N^{+}}{2}, B=\frac{N-N^{+}}{2i}$ and
\[
    p(x,y) = (x+iy)(x-iy) - 1 = x^2 + y^2 - 1 \in \bb R[x,y] \,.
\]
Therefore, according to \thref{definidef} any unitary operator on a Krein space 
is definitizable normal, and the ideal $I$ generated by all real
definitizing polynomials always contains $p(x,y)$.
Since the ideal generated by $p$ is not zero-dimensional, 
the zero-dimensionality of $I$ implies the existence 
a definitizing polynomial different from $p$.

\begin{remark}
  If, for example, there exists a polynomial $a\in \bb C[z]\setminus \{0\}$
  such that $[a(N)u,u] \geq 0, \ u \in \mc K$, then the ideal $J$
  generated by $a$ (considered as a polynomial in $\bb C[z,w]$) and 
  $b(z,w)=zw-1$ in $\bb C[z,w]$ is zero-dimensional. Indeed, it is easy to see that
  the set $V(J)$ of common zeros of $a$ and $b$ is finite, which by \cite{CLO2}, page 39,
  implies zero-dimensionality.
  Since $c(z,w)\mapsto c(x+iy,x-iy)$ constitutes an isomorphism from 
  $\bb C[z,w]$ onto $\bb C[x,y]$, also the ideal generated by 
  $a(x+iy)$ and $p(x,y)$ in $\bb C[x,y]$ is zero-dimensional. Hence, 
  the same is true for $I$, and we can apply the functional 
  calculus developed Section \ref{Funcalzero}. 
\end{remark}

Assume that $N\in B(\mc K)$ is unitary and that the ideal $I$ generated by all real
definitizing polynomials is zero-dimensional. Consequently, we can apply the functional 
calculus developed in Section \ref{Funcalzero}. From $p\in I$ we conclude
\[
    a \in V(I) \Rightarrow p(a) = 0 \,.
\]
Hence, the elements of $V_{\bb R}(I)$ are contained in $\bb T$, and
$(\xi,\eta) \in V(I) \setminus \bb R^2$ yields 
\[
    (\xi+i\eta)\overline{(\bar\xi+i\bar\eta)} = \xi^2+\eta^2=1.
\]
Moreover,
with $N$ also $\Theta(N)$ is unitary in the Hilbert space
$\mc H$; see \thref{defNspaces} and \eqref{thetaVdef}. In particular,
$\sigma(\Theta(N)) \subseteq \bb T$. From \thref{sigmaN} we derive that
$\sigma(N)$ is contained in $\bb T$ up to finitely many points which are located 
in $\bb C \setminus \bb T$ symmetric with respect to $\bb T$.

\section{Transformations of definitizable normal operators}

In this final section we examine, whether basic transformations, such as $\alpha N, N + \beta I_{\mc K}, N^{-1}$
with $\alpha, \beta \in\bb C, \ \alpha\neq 0$, 
of definitizable normal operators $N$ are again definitizable, and how the corresponding ideals $I$
behave.

For $\beta\in\bb C$ it is easy to see that $p(x,y)$ is a real definitizing polynomial
for $N$ if and only if $p(x-\Re \beta,y-\Im \beta)$ is real definitizing for $N + \beta I_{\mc K}$.
Since $r(x,y) \mapsto r(x-\Re \beta,y-\Im \beta)$ is a ring automorphism on $\bb C[x,y]$,
the respective ideals $I$, corresponding to $N$ and $N + \beta I_{\mc K}$, are zero-dimensional, or not, 
at the same time.
    
Similarly, $p(x,y)$ is a real definitizing polynomial
for $N$ if and only if $p(x \Re \frac{1}{\alpha} - y \Im \frac{1}{\alpha}, x\Im \frac{1}{\alpha} + y \Re \frac{1}{\alpha})$
is real definitizing for $\alpha N$. Also $r(x,y) \mapsto 
r(x \Re \frac{1}{\alpha} - y \Im \frac{1}{\alpha}, x\Im \frac{1}{\alpha} + y \Re \frac{1}{\alpha})$ 
is a ring automorphism on $\bb C[x,y]$. Hence, the ideal $I$ corresponding to $N$ is zero-dimensional if and only if
the ideal $I$ corresponding to $\alpha N$ is zero-dimensional.

For the inverse $N^{-1}$ the situation is more complicated. We formulate two
results that we will need. 
The first assertion is straight forward to verify. We omit its proof.

\begin{lemma}\thlab{trafos2}
    The mapping $\Phi: p(x,y) \mapsto p(\frac{z+w}{2},\frac{z-w}{2i})$ from $\bb C[x,y]$ to
    $\bb C[z,w]$ is an isomorphism, where $p$ is real, i.e.\ $p(\bar x, \bar y) = \overline{p(x,y)}$, 
    if and only if $\overline{\Phi(p)(z,w)} = \Phi(p)(\bar w,\bar z)$. 
\end{lemma}

Obviously, for a normal $N=A+iB$ and $p(x,y)\in \bb C[x,y]$ we have
\begin{equation}\label{Phiwitz}
    p(A,B) = \Phi(p)(N,N^+) \,.
\end{equation}

For a polynomial $q\in\bb C[z,w]\setminus\{0\}$ let $d(q)$ be the maximum of
    the $z$-degree of $q$ and the $w$-degree of $q$. Moreover, we set
\[
    \varpi(q)(z,w):= (zw)^{d(q)} q(\frac{1}{z},\frac{1}{w}) \in \bb C[z,w] \,.
\]

\begin{lemma}\thlab{trafos3}
    If $I=\langle q_1,\dots,q_m \rangle$ is zero-dimensional with 
    polynomials $q_1,\dots,q_m$ such that $\overline{q_j(z,w)} = q_j(\bar w,\bar z)$, then
    the ideal $\langle \varpi(q_1),\dots,\varpi(r_m) \rangle$ is also zero-dimensional.
\end{lemma}
\begin{proof}
    Let $(\zeta,\eta) \in V(\varpi(q_1),\dots,\varpi(r_m))$. For $\zeta\neq 0 \neq \eta$ we conclude
    $q_j(\frac{1}{\zeta},\frac{1}{\eta}) = 0, \ j=1,\dots, m$, and in turn 
    $(\zeta,\eta) \in \{ (z,w) \in (\bb C\setminus\{0\})^2: (\frac{1}{z},\frac{1}{w}) \in V(I) \}$.
    
    Assume that $\eta=0$ and $\zeta\neq 0$. If $q_j(z,w) = \sum_{k,l=0}^{d(q_j)} b_{k,l} z^k w^l$, then
    $\overline{q_j(z,w)} = q_j(\bar w,\bar z)$ yields $b_{k,l} = \bar b_{l,k}$,
    and we have $\varpi(q_j)(z,w) = \sum_{k,l=0}^{d(q_j)} b_{d(q_j)-k,d(q_j)-l} z^k w^l$. According to the choice
    of $d(q_j)$ and by $b_{k,l} = \bar b_{l,k}$ the polynomial
    \[
	\rho_j(z):=\varpi(q_j)(z,0) =  \sum_{k=0}^{d(q_j)} b_{d(q_j)-k,d(q_j)} z^k 
    \]
    is non-zero and satisfies $\rho_j(\zeta)=0$, i.e.\ $(\zeta,\eta) \in \rho_j^{-1}(\{0\})\times\{0\}$.
    
    From $\overline{q_j(z,w)} = q_j(\bar w,\bar z)$ we conclude $\rho_j(\bar w) = \overline{\varpi(q_j)(0,w)}$.
    Hence, $\zeta=0$ and $\eta\neq 0$ yields 
    $(\zeta,\eta) \in \{0\}\times \overline{\rho_j^{-1}(\{0\})}$.
    
    In any case $(\zeta,\eta)$ is contained in
    \begin{multline*}
	\{(0,0)\} \cup \{ (z,w) \in (\bb C\setminus\{0\})^2: (\frac{1}{z},\frac{1}{w}) \in V(I) \} \cup \\
		\cup \bigcap_{j=1,\dots,m} \rho_j^{-1}(\{0\})\times\{0\}
		\cup \bigcap_{j=1,\dots,m} \{0\}\times \overline{\rho_j^{-1}(\{0\})} \,.
    \end{multline*}
    Consequently, $V(\varpi(q_1),\dots,\varpi(r_m))$ is finite, and in turn $\langle\varpi(q_1),\dots,\varpi(r_m) \rangle$ is zero-dimensional; see
    \cite{CLO2}, page 39.
\end{proof}

\begin{proposition}\thlab{invdefbar}
    Let $N$ be normal and bijective on the Krein space $\mc K$.
    If $p(x,y)$ is real definitizing for $N$, then 
    $\Phi^{-1}\Big(\varpi\big(\Phi(p)\big)\Big)$ is definitizing for $N^{-1}$.
    Moreover, if the ideal $I$ generated by all real definitizing $p(x,y)$ for $N$ is zero-dimensional, then
    also the ideal generated by all real definitizing polynomials for $N^{-1}$ is zero-dimensional.
\end{proposition}
\begin{proof}
    Let $p(x,y)$ be real definitizing for $N$. 
    By \thref{trafos2} we have $\overline{\Phi(p)(z,w)} = \Phi(p)(\bar w,\bar z)$, and in turn
    $\overline{\varpi(\Phi(p))(z,w)} = \varpi(\Phi(p))(\bar w,\bar z)$. 
    We write $\Phi(p)(z,w) = \sum_{k,l=0}^{d(\Phi(p))} b_{k,l} z^k w^l$, and consequently
    $\varpi(\Phi(p))(z,w) = \sum_{k,l=0}^{d(\Phi(p))} b_{d(\Phi(p))-k,d(\Phi(p))-l} z^k w^l$. 
    
    By \eqref{Phiwitz} for $u\in\mc K$ we have
    \begin{align*}
	[\Phi^{-1}\Big(\varpi\big(\Phi(p)\big) & (\Re N^{-1}, \Im N^{-1}) u,u] = 
	[\varpi(\Phi(p))(N^{-1},N^{-+}) u,u] \\ & = 
	[\sum_{k,l=0}^{d(\Phi(p))} b_{d(\Phi(p))-k,d(\Phi(p))-l} (N^{-1})^k (N^{-+})^l u,u] \\ & = 
	[\Phi(p)(N,N^+) \, (N^{-1})^{d(\Phi(p))} u,(N^{-1})^{d(\Phi(p))} \, u] \\ & = [p(A,B) \, (N^{-1})^{d(\Phi(p))} u,(N^{-1})^{d(\Phi(p))} \, u] \geq 0 \,.
    \end{align*}
    Hence, $\Phi^{-1}\Big(\varpi\big(\Phi(p)\big)\Big)$ is real definitizing for $N^{-1}$.
    Finally, if $I$ is zero-dimensional and generated by real definitizing $p_1,\dots,p_m$, then 
    $\Phi(I)=\langle \Phi(p_1),\dots,\Phi(p_m) \rangle$ is zero-dimensional in $\bb C[z,w]$. According to \thref{trafos3}
    $\langle \varpi\big(\Phi(p_1)\big),\dots,\varpi\big(\Phi(p_m)\big) \rangle$, and hence
    also $\langle \Phi^{-1}\Big(\varpi\big(\Phi(p_1)\big)\Big),\dots,\Phi^{-1}\Big(\varpi\big(\Phi(p_m)\big)\Big) \rangle$ is zero-dimensional.
    Since its generators are real definitizing for $N^{-1}$ also
    the ideal generated by all real definitizing polynomials for $N^{-1}$ is zero-dimensional.
\end{proof}

\end{document}